\documentclass[reqno]{amsart}
\usepackage{amsmath,amssymb,amsthm}
\usepackage{thmtools}
\usepackage{mathptmx}      
\usepackage{mathrsfs}
\usepackage{latexsym}
\usepackage{amssymb}
\usepackage{bm}
\usepackage[usenames,dvipsnames]{xcolor}
\usepackage{graphicx,pifont}
\usepackage{multicol,multirow}
\usepackage{ulem,cancel}
\usepackage{tikz}
\usetikzlibrary[patterns]
\usepackage{comment}
\usepackage{array}
\usepackage{subcaption}
\usepackage{txfonts}
\usepackage{mathtools}
\usepackage[shortlabels]{enumitem}
\usepackage{hyperref}
\usepackage{booktabs}

\usepackage{cleveref}
\crefname{equation}{}{}
\crefname{enumi}{}{}
\crefname{figure}{Figure}{Figure}
\crefname{table}{Table}{Table}
\crefname{subsection}{Subsection}{Subsections}
\crefname{lemma}{Lemma}{Lemma}
\crefname{theorem}{Theorem}{Theorem}
\crefname{proposition}{Proposition}{Proposition}
\crefname{section}{Section}{Section}
\crefname{appendix}{Appendix}{Appendix}

\newtheorem{theorem}{Theorem}[section]
\newtheorem{lemma}[theorem]{Lemma}

\theoremstyle{definition}

\theoremstyle{remark}
\newtheorem{remark}[theorem]{Remark}
\numberwithin{equation}{section}
\numberwithin{figure}{section}
\numberwithin{table}{section}
\newtheorem{conjecture}[theorem]{Conjecture}

\newcommand{\oname}[1]{\mathrm{#1}}
\newcommand{\abs}[1]{\left|#1\right|}

\newcommand{\eqdef}{\stackrel{\mathrm{def}}{=\joinrel=}}

\newcommand{\phf}[1]{{{#1}+1/2}}
\newcommand{\phfg}[2]{{{#1}+{#2}/2}}

\newcommand{\mhf}[1]{{{#1}-1/2}}

\newcommand{\mhfg}[2]{{{#1}-{#2}/2}}
\newcommand{\floor}[1]{\left\lfloor{#1}\right\rfloor}

\newcommand{\fdm}{{\textrm{fdm}}}

\newcommand{\cm}{\textrm{\ding{51}}}
\newcommand{\xm}{\textrm{\ding{55}}}
\newcommand{\hc}{\mathrlap{\cm}{~{\large\backprime}}}

\newcommand{\oone}{\varepsilon}

\begin{document}

\title[Stability Barrier of Hermite Type Methods]{On the Stability Barrier of Hermite Type Discretizations of Advection Equations}

\author[X.~Zeng]{Xianyi Zeng}
\address{Department of Mathematics, Lehigh University, Bethlehem, PA 18015, United States.}
\email[Corresponding author, X.~Zeng]{xyzeng@lehigh.edu}

\date{\today}



\begin{abstract}
  In this paper we establish a stability barrier of a class of high-order Hermite-type discretization of 1D advection equations underlying the hybrid-variable (HV) and active flux (AF) methods.
  These methods seek numerical approximations to both cell-averages and nodal solutions and evolves them in time simultaneously.
  It was shown in earlier work that the HV methods are supraconvergent, providing that the discretization uses more unknowns in the upwind direction than the downwind one, similar to the ``upwind condition'' of classical finite-difference schemes.
  Although it is well known that the stencil of finite-difference methods could not be too biased towards the upwind direction for stability consideration, known as ``stability barrier'', such a barrier has not been established for Hermite-type methods.
  In this work, we first show by numerical evidence that a similar barrier exists for HV methods and make a conjecture on the sharp bound on the stencil.
  Next, we prove the existence of stability barrier by showing that the semi-discretized HV methods are unstable given a stencil sufficiently biased towards the upwind direction.
  Tighter barriers are then proved using combinatorical tools, and finally we extend the analysis to studying other Hermite-type methods built on approximating nodal solutions and derivatives, such as those widely used in Hermite WENO methods.
\end{abstract}

\maketitle

\section{Introduction}
\label{sec:intro}
In this work we establish stability barriers of Hermite-type discretizations of linear advection equations, a prototype problem for general hyperbolic conservation laws.
Hyperbolic conservation laws are underlying many systems such as those related to fluid mechanics, cellular dynamics in tumor growth models, and population migrations.
While most practical methods focus on second-order accuracy, largely due to the existence of shock waves, recent interests in multi-scale problems make higher-order methods attractive as they tend to resolve small structures better in regions away from discontinuities.

In general, higher-order methods either demand a larger stencil to update the solution at a local node or cell, which makes parallel computation more difficult, or they insert more degrees of freedom inside one mesh cell, which results in more unknowns to achieve the same order of accuracy.
Examples of the first kind include the high-order finite difference methods~\cite{AIserles:1984b,GAGerolymos:2009a,JBerg:2012a} and k-exact finite volume schemes~\cite{TJBarth:1993a,TBarth:2004a,CMichalak:2009a}, whereas the latter is exemplified by discontinuous Galerkin methods~\cite{BCockburn:1989a,WCao:2014a}.
There is a vast volume of literature in any of these topics and our lists are by no means complete.
Recently there have been growing interests in methods utilizing multiple moments (such as cell-average and nodal values, or nodal values and derivatives, or all of them), which seek a balance between using a too-large stencil and too many unknowns in each geometric cell.
This work focuses on methods that use two moments, particularly the active flux methods~\cite{WBarsukow:2021a,RAbgrall:2023a,JDuan:2025a} and the hybrid-variable methods developed by the author~\cite{XZeng:2014a,XZeng:2019a,MMHasan:2023a,XZeng:2024a}, which find numerical approximations to both cell-averages and nodal values; methods using nodal values and derivatives like the Hermite WENO methods~\cite{GCapdeville:2008a,HLiu:2015a} and multi-moment methods~\cite{MRNorman:2012a,XLi:2020a} will also be discussed after presenting the main result.
It should be noted that methods involving more moments have been proposed in the literature, such as~\cite{JGoodrich:2006a,DAppelo:2012a,XChen:2014a}; we shall leave the analysis of these methods in future work.

To handle discontinuities, these methods combine a high-order discretization of the equation and a non-linear adaptive dection of discontinuities and stability enhancements near these locations.
For example, HV methods utilize a residual-consistent artificial viscosity in compressible flow computations~\cite{XZeng:2024a}, the AF methods use the classical idea of flux limiting and special local interpolation to reduce undershoots or overshoots~\cite{WBarsukow:2021a,JDuan:2025a}, and both Hermite WENO methods and multi-moment schemes cited before adopt WENO-type limiters to build a high-order accurate flux as a convex combination of lower-order ones.
In this work we avoid the discussion of nonlinear stability enhancements and concentrate solely on the underlying high-order discretization that is typically linear for linear problems.
The rationale is that almost all aforementioned methods reduce to their underlying high-order ones when the solution is sufficiently smooth; thus the linear stability of this high-order scheme is necessary for successful application of the method to solve more complicated problems.
As the underlying linear schemes are identical for active flux and hybrid-variable methods, we shall refer to it as the HV discretizations for simplicity, as most of the notations are adopted from the author's previous work~\cite{XZeng:2019a}.

To this end, the main contribution of this paper is shedding lights on the condition under which an optimal HV discretization gives rise to a linearly stable scheme; and our analysis is conducted in the context of method of lines as one can always pair a stable semi-discretized method with appropriate time-integrator (explicit or implicit) to construct a lienarly stable fully-discretized scheme.
Note that a similar question has been studied long ago to the single-moment finite-difference methods (FDM), which states that FDMs with optimal accuracy for an advection equation is linearly stable if and only if $r\le l\le r+2$, where $l$ and $r$ are number of points in the upwind and downwind directions, respectively~\cite{AIserles:1982a,AIserles:1983a,BDespres:2009a}.
Note that $l=r$ corresponds to a central scheme in which case the $L^2$-norm of the solution is preserved, and it is sometimes referred to as neutrally stable.
It is usually adopted that $r+1\le l\le r+2$ in practice.
This inequality is known as the {\it stability barrier} of FDM to solve linear advection equations; and our purpose is to show a similar barrier exists to HV discretizations.

Extension of existing analysis of stability barrier from FDMs to HVs is highly challenging.
In a series of work of Iserles and coworkers, the theory of order stars is developed~\cite{AIserles:1982a,AIserles:1985a,AIserles:1985b,AIserles:1985c,AIserles:1991a} and it provides a geometrical connection between the accuracy and the stability of a finite-difference discretization.
Later, Despr\'{e}s~\cite{BDespres:2008a,BDespres:2009a} proved the same stability barrier using elementary techniques and studying the integral form of the local truncation error.
However, neither technique seems to extend to Hermite-type methods as discussed briefly in~\cref{app:rev}.

To this end, our strategy is to center around the classical Fourier analysis (or von Neumann stability analysis) and use heavily combinatoric tools to study the characteristic trigonometric function that determines the stability of the methods.
In particular, after reviewing the coefficients and charateristic functions of the HV discretizations in~\cref{sec:prelim}, our main results on stability barriers are presented in~\cref{sec:bar}.
The results in this section include two parts:~\cref{thm:bar_sqrt} proves the existence of the stability barrier, although it does not provide a tight bound on where the barrier is, and~\cref{thm:bar_pi} provides instability results when the stencil is only moderately upwind-biased, which is of more practical importance.
Some of the necessary combinatoric tools involved in the proofs of this section are provided in the appendices.
Next, we show that the techniques established in~\cref{sec:bar} can be utilized to study the stability barrier of other Hermite-type methods, particularly the ones using nodal values and nodal derivatives, see~\cref{sec:hermite}.
Lastly,~\cref{sec:concl} concludes the current paper and discuss future directions.




\section{Preliminaries}
\label{sec:prelim}
Let us consider the Cauchy problem for the advection equation:
\begin{equation}\label{eq:prelim_adv}
  u_t + u_x = 0\;, 
\end{equation}
on the domain $[0,\,1]$ with periodic boundary condition $u(0,t)=u(1,t)$; and we seek numerical approximations with respect to a uniform grid $x_j=jh$, where $h=1/N>0$ is the grid size.
The discretization of~\cref{eq:prelim_adv} in space by the hybrid-variable (HV) method concerns approximations to both the nodal values $u_j(t)\approx u_j^\ast(t) \eqdef u(x_j,t)$ and cell averages $\overline{u}_{\phf{j}}(t)\approx \overline{u}^\ast_{\phf{j}} \eqdef \frac{1}{h}\int_{x_j}^{x_{j+1}}u(x,t)dx$.
It is understood that $u_j=u_{j+N}$ and $\overline{u}_{\phf{j}}=\overline{u}_{\phf{j+N}}$ for all integer $j$.
Here the asteroid denotes exact values.
An HV scheme is determined by the discrete differential operator (DDO) $[\mathcal{D}_x]$, which is defined by (suppressing the dependence on $t$):
\begin{equation}\label{eq:prelim_ddo}
  [\mathcal{D}_xu]_j \eqdef \frac{1}{h}\sum_{k=-l}^{r-1}\alpha_k\overline{u}_{\phf{j+k}}+\frac{1}{h}\sum_{k=-l'}^{r'}\beta_ku_{j+k}\;.
\end{equation}
The construction of $[\mathcal{D}_xu]_j\approx u_x(x_j)$ uses $l$ cells and $l'$ nodes to the left of $x_j$, and $r$ cells and $r'$ nodes to the right of $x_j$; the quadruple $(l,r,l',r')$ is thus called the stencil of the DDO $[\mathcal{D}_x]$.
It is natural to assume that the stencil utilizes a continuous set of entities, thus $l'\le l\le l'+1$ and $r'\le r\le r'+1$.
An equivalent way to denote such $[\mathcal{D}_x]$ is the pair of indices $(L,R)$, where $L=l+l'$ and $R=r+r'$; one can determine the four-integer index by $l'=\floor{L/2}$, $l=L-l'$, $r'=\floor{R/2}$, and $r=R-r'$.

The DDO~\cref{eq:prelim_ddo} is $p$-th order accurate if for sufficiently smooth $u$, one has:
\begin{displaymath}
  [\mathcal{D}_xu^\ast]_j = u_x(x_j) + c_p\frac{d^{p+1}u(x_j)}{dx^{p+1}}h^p + O(h^{p+1})
\end{displaymath}
for some $c_p\ne0$.
In previous work~\cite{XZeng:2019a} it is shown that given a continuous stencil $(l,r,l',r')$ there exists a unique $[\mathcal{D}_x]$ that has the optimal order $p=l+r+l'+r'=L+R$.
We shall only consider these optimally accurate DDO in the present work.
Once $[\mathcal{D}_x]$ is chosen, the semi-discretization of~\cref{eq:prelim_adv} is thus given by for all $1\le j\le N$:
\begin{equation}\label{eq:prelim_semi}
  \left\{\begin{array}{l}
    \overline{u}_{\phf{j}}' + \frac{1}{h}\left(u_{j+1}-u_j\right) = 0\;, \\ \vspace*{-.05in} \\
    u_j' + [\mathcal{D}_xu]_j = 0\;.
  \end{array}\right.
\end{equation}
By the von Neumann analysis, much of the property of the numerical method can be obtained by studying the eigenvalues of this linear system of ordinary differential equations.
Particularly, one can show that (for example, following~\cite[Section 4]{XZeng:2019a}) all eigenvalues of the ODE system~\cref{eq:prelim_semi} lie on the set of eigenvalues (denoted $\mathcal{S}$) of the $2\times2$ matrix:
\begin{equation}\label{eq:prelim_semi_mat}
  \begin{bmatrix}
    0 & e^{i\theta}-1 \\
    G(\theta) & H(\theta)
  \end{bmatrix}\;,\quad\textrm{ where }
  G(\theta) = \sum_{k=-l}^{r-1}\alpha_k e^{ik\theta}\;,\ \textrm{ and }\  
  H(\theta) = \sum_{k=-l'}^{r'}\beta_k e^{ik\theta}\;,
\end{equation}
with $\theta$ runs from $0$ to $2\pi$. 
The characteristic polynomial of this matrix is given by:
\begin{equation}\label{eq:prelim_semi_char}
  \lambda^2 - H(\theta)\lambda - (e^{i\theta}-1)G(\theta) = 0\;,
\end{equation}
and the set $\mathcal{S}$ is thus defined as:
\begin{equation}\label{eq:prelim_semi_eigset}
  \mathcal{S} = \{\lambda\in\mathbb{C}:\, \lambda^2-H(\theta)\lambda-F(\theta) = 0\ \textrm{ for some }\ 0\le\theta\le2\pi\},
\end{equation}
where $F(\theta)=(e^{i\theta}-1)G(\theta)$.
Physically, $\theta=\kappa h$ where $\kappa$ represents the wave number.

It was proved that given sufficiently smooth initial data, the semi-discretized solution converges to the actual one pointwise at $(p+1)$-th order of accuracy if: (1) the DDO $[\mathcal{D}_x]$ is $p$-th order, (2) the stencil is biased towards the upwind direction, or $L>R$, and (3) the ODE system is stable (to be made precise later).
Here the first condition is known as {\it supraconvergence}, as the order of the method is higher than that of the local truncation error.
The upwind-biased condition, despite its similarity to that of finite-difference method, is not due to stability considerations but a necessary condition for one of the two roots to~\cref{eq:prelim_semi_char} decays exponentially fast to zero as $h\to0$ or $\theta\to0$.
This paper focuses on the third condition and investigates the stability given an upwind biased stencil ($L>R$).

When saying the linear ODE system~\cref{eq:prelim_semi} is stable, we mean all its eigenvalues have negative real parts except at most a simple one at zero.
Note that these eigenvalues are precisely $-\lambda(\theta)$, where $\lambda$ is a root of~\cref{eq:prelim_semi_char}; and it is easy to compute when $\theta=0$, the two roots of~\cref{eq:prelim_semi_char} are $0$ and $H(0)=\sum_{k=-l'}^{r'}\beta_k>0$\footnote{The two conditions $H(0)>0$ and $L>R$ are equivalent, see~\cite[Proposition 5]{XZeng:2019a}, as well as the discussion in the later part of this paper}, it is not difficult to see that stability for all admissible combinations of wavenumber, domain size, and number of cells is equivalent to say $\mathcal{S}'$ is contained in the open right complex plane $\mathbb{C}^+=\{z:\,\oname{Re}z>0\}$, where:
\begin{equation}\label{eq:prelim_semi_traj}
  \mathcal{S}' = \{\lambda\in\mathbb{C}:\, \lambda^2-H(\theta)\lambda-F(\theta) = 0\ \textrm{ for some }\ 0<\theta<2\pi\}.
\end{equation}

Our analysis is based on the following sufficient and necessary condition for $\mathcal{S}'\subseteq\mathbb{C}^+$.
\begin{lemma}\label{lm:prelim_stab}
  The set $\mathcal{S}'$ is contained in $\mathbb{C}^+$ if and only if for all $0<\theta<2\pi$, the following two conditions are satisfied:
  \begin{subequations}\label{eq:prelim_stab_cond}
    \begin{align}
      \label{eq:prelim_stab_cond_h}
      \oname{Re}H > 0\;, \\
      \label{eq:prelim_stab_cond_g}
      \oname{Re}H\,\oname{Re}(\overline{H}F)+(\oname{Im}F)^2 < 0\;.
    \end{align}
  \end{subequations}
  Here $H=H(\theta)$ and $F=F(\theta)$.
\end{lemma}
\begin{proof}
  Let the two roots of~\cref{eq:prelim_semi_char} be $\lambda_{1,2}=x_{1,2}+iy_{1,2}$, then $H=z_1+z_2$ and $F=-z_1z_2$; and one can compute:
  \begin{align*}
    &\oname{Re}H = x_1 + x_2\;, \\
    &\oname{Re}H\,\oname{Re}(\overline{H}F)+(\oname{Im}F)^2 = - x_1x_2[(x_1+x_2)^2+(y_1-y_2)^2]\;.
  \end{align*}
  It is easy to see that $x_1,x_2>0$ if and only if~\cref{eq:prelim_stab_cond} hold.
\end{proof}
\begin{remark}
  The seemingly simple proof does not tell how the conditions~\cref{eq:prelim_stab_cond} are derived in the first place.
  In fact, one can get them as sufficient conditions for stability using a classical result by Evelyn Frank~\cite[Theorem 2.1]{EFrank:1947a}; and we show here that they are also necessary.
\end{remark}

The question about the stability barrier concerns whether a DDO $[\mathcal{D}_x]$ with optimal order of accuracy and a stencil such that $L>R$ will be stable.
In this regard, this paper has two purposes: first we want to show that such a stability barrier exists for HV methods, that is, not all upwind HV methods are stable; and then prove some results concerning where the barrier actually is.
First of all, the condition~\cref{eq:prelim_stab_cond} in combination with explicit formula for $[\mathcal{D}_x]$ (see~\cite[Theorem 1]{XZeng:2019a} or stated here as~\cref{lm:prelim_coef} below) allow one to verify the stability (or instability) of any particular HV scheme.
\begin{lemma}\label{lm:prelim_coef}
  The unique DDO $[\mathcal{D}_x]$ with stencil $(l,r,l',r')$ and optimal order of accuracy $p=l+r+l'+r'$ is given by~\cref{eq:prelim_ddo} and:
  \begin{subequations}\label{eq:prelim_coef}
    \begin{align}
      \label{eq:prelim_coef_cell_neg}
      \alpha_\nu &= -(1-\delta_{ll'})\frac{2}{l^2}C^{l,r}_{-l}C^{l,r'}_{-l}-\sum_{k=-l'}^\nu\frac{2(1+k(\zeta_k^{l,r}+\zeta_k^{l',r'}))}{k^2}C_k^{l,r}C_k^{l',r'},\ \forall -l\le\nu<0\;, \\
      \label{eq:prelim_coef_cell_pos}
      \alpha_\nu &= \sum_{k=\nu+1}^{r'}\frac{2(1+k(\zeta_k^{l,r}+\zeta_k^{l',r'}))}{k^2}C_k^{l,r}C_k^{l',r'}+(1-\delta_{rr'})\frac{2}{r^2}C^{l,r}_{r}C^{l',r}_{r},\ \forall 0\le\nu<r\;, \\
      \label{eq:prelim_coef_node_zero}
      \beta_0 &= 2(\zeta_0^{l,r}+\zeta_0^{l',r'})\;, \\
      \label{eq:prelim_coef_node_nz}
      \beta_\nu &= -\frac{2}{\nu}C^{l,r}_{\nu}C^{l',r'}_{\nu}\;,\ \forall -l'\le\nu\le r',\ \nu\ne0\;.
    \end{align}
  \end{subequations}
  Here $\delta_{ab}$ is Kronecker delta, $C^{m,n}_k=\frac{m!n!}{(m+k)!(n-k)!}$, and $\zeta_k^{m,n} = H_{m+k}-H_{n-k}$ for all $-m\le k\le n$, where $H_k=1+\cdots+1/k$ is the $k$-th Harmonic number and it is understood that $H_0=0$.
\end{lemma}
The stability of HV methods with stencil $(L,R)$, $0\le R<L\le8$ are checked and listed in~\cref{tb:prelim_stab}.
\begin{table}\centering
  \caption{Stability of semi-discretized HV methods: $\cm$ -- stable, $\hc$ -- unstable but satisfies~\cref{eq:prelim_stab_cond_h}, $\xm$ -- unstable and violates~\cref{eq:prelim_stab_cond_h}.}
  \label{tb:prelim_stab}
  \begin{tabular}{@{}lllllllll@{}}
    \toprule[.5mm]
        & \multicolumn{7}{c}{$R$} \\ \cmidrule[.2mm]{2-9}
    $L$ & 0     & 1     & 2     & 3     & 4     & 5     & 6     & 7     \\ \cmidrule[.3mm]{1-9}
    1   & $\cm$ &       \\
    2   & $\cm$ & $\cm$ &     \\
    3   & $\cm$ & $\cm$ & $\cm$ &       \\
    4   & $\xm$ & $\hc$ & $\cm$ & $\cm$ &       \\
    5   & $\xm$ & $\xm$ & $\hc$ & $\cm$ & $\cm$ &       \\
    6   & $\xm$ & $\xm$ & $\xm$ & $\hc$ & $\cm$ & $\cm$ &       \\
    7   & $\xm$ & $\xm$ & $\xm$ & $\xm$ & $\hc$ & $\cm$ & $\cm$ &       \\
    8   & $\xm$ & $\xm$ & $\xm$ & $\xm$ & $\xm$ & $\hc$ & $\cm$ & $\cm$ \\
    \bottomrule[.4mm]
  \end{tabular}
\end{table}
For example, we may consider three $8$th order methods, constructed by $7$th order DDOs with stencils $(4,3)$, $(5,2)$, and $(7,0)$, respectively: 
\begin{itemize}
  \item The $(4,3)$-DDO is given by:
    \begin{equation}\label{eq:prelim_ddo_43}\small
      [\mathcal{D}_xu]_j = \frac{-53\overline{u}_{\mhfg{j}{3}}-725\overline{u}_{\mhf{j}}+355\overline{u}_{\phf{j}}+3\overline{u}_{\phfg{j}{3}}}{216h}+\frac{u_{j-2}+24u_{j-1}+18u_j-8u_{j+1}}{18h}\;,
    \end{equation}
    for which~\cref{eq:prelim_stab_cond} are:
    \begin{align*}
      &\oname{Re}H = \frac{1}{9}\left(\cos\theta+1\right)\left(\cos\theta+7\right)+\frac{1}{6}>0\;, \\
      &\oname{Re}H\,\oname{Re}(\overline{H}F)+(\oname{Im}F)^2 = -\frac{1}{1458}(1-\cos\theta)^5(13+\cos\theta)<0\;.
    \end{align*}
  \item The $(5,2)$-DDO is given by:
    \begin{equation}\label{eq:prelim_ddo_52}\small
      [\mathcal{D}_xu]_j = \frac{-\overline{u}_{\mhfg{j}{5}}-77\overline{u}_{\mhfg{j}{3}}-401\overline{u}_{\mhf{j}}+59\overline{u}_{\phf{j}}}{72h}+\frac{2u_{j-2}+18u_{j-1}+16u_j-u_{j+1}}{6h}\;.
    \end{equation}
    For this method:
    \begin{align*}
      &\oname{Re}H = \frac{2}{3}(\cos\theta+1)(\cos\theta+\frac{13}{4})+\frac{1}{6}>0\;, \\
      &\oname{Re}H\,\oname{Re}(\overline{H}F)+(\oname{Im}F)^2 = \frac{1}{324}(1-\cos\theta)^5(2+5\cos\theta)\nless0\;.
    \end{align*}
  \item The $(7,0)$-DDO is given by:
    \begin{equation}\label{eq:prelim_ddo_70}\small
      [\mathcal{D}_xu]_j = \frac{-\overline{u}_{\mhfg{j}{7}}-65\overline{u}_{\mhfg{j}{5}}-209\overline{u}_{\mhfg{j}{3}}-145\overline{u}_{\mhf{j}}}{8h}+\frac{8u_{j-3}+108u_{j-2}+144u_{j-1}+47u_j}{6h}\;,
    \end{equation}
    and one computes:
    \begin{displaymath}
      \oname{Re}H = \frac{16}{3}\cos^3\theta+36\cos^2\theta+20\cos\theta-\frac{61}{6}\;,
    \end{displaymath}
    which is clearly negative when $\theta$ is near $\pi/2$.
\end{itemize}
To this end, we have the following conjecture on the stability barrier of HV methods:
\begin{conjecture}\label{cj:prelim_sb}
  The semi-discretized HV method is stable if and only if $0\le R < L < R+3$ except if $(L,R)=(3,0)$, when the method is stable. 
\end{conjecture}
The conjecture is very similar to that of the finite-difference method in the literature: an optimally accurate finite-difference method is stable if and only if $R<L<R+3$.
It was first proved using the theory of order stars~\cite{AIserles:1982a,AIserles:1983a} and then by a more elementary method analyzing the integral form of an error function~\cite{BDespres:2009a}.
However, neither strategy seems to work in the case of hybrid-variable methods, and we explain the challenge using these two approaches in~\cref{app:rev}.

\section{Barriers of stable HV methods}
\label{sec:bar}
In this section, we establish some barrier results for stable semi-discretized methods.
The main result is given in~\cref{thm:bar_sqrt}, which confirms that given any $R$, an HV method with stencil $(L,R)$ cannot be stable for large $L$. 
Tighter barriers in special cases are given later in~\cref{thm:bar_pi} and~\cref{thm:bar_asymp}.
\begin{theorem}\label{thm:bar_sqrt}
  Let $[\mathcal{D}_x]$ be a DDO with stencil $(L,R)$ and optimal accuracy, then if $\oname{Re}H(\theta)>0$ for all $0<\theta\le\pi$, there is:
  \begin{equation}\label{eq:bar_sqrt}
    L-R < \min(2R+7,\,9+\sqrt{21R+49})\;.
  \end{equation}
\end{theorem}
This theorem indicates that if~\cref{eq:bar_sqrt} is violated, the corresponding semi-discretized HV method is unstable.
Note that in~\cref{eq:bar_sqrt}, the linear bound is tighter than the square-root one when $R\le9$.
The proof is based on analyzing whether the $\beta$-coefficients satisfy the boundedness condition on coefficients of non-negative trigonometric polynomials, as given in~\cref{sec:bar_sqrt}.

While this theorem confirms that given a downwind stencil $R$, the HV method cannot be stable if the upwind stencil $L$ is too large; the bound is relatively loose, in the view of~\cref{tb:prelim_stab}.
Analysing HV methods with small stencil bias (i.e., $L-R$) is more relevant for practical purposes, as they tend to simplify inter-processor communications in parallel computations.
These methods, however, are much more difficult to analyse; and in what follows we present some results by investigating the sign of $\oname{Re}H(\pi)$.
For this purpose, we develop a combinatorics tool to compute $\oname{Re}H(\theta)$ in~\cref{sec:bar_reh} and have the following instability result.
\begin{theorem}\label{thm:bar_pi}
  The HV method is unstable if $4\le L-R\le 7$.
\end{theorem}

The proof of this theorem is given in~\cref{sec:bar_pi}, where we can see the complexity of the computation increases dramatically as $L-R$ increases.
Unfortunately, effort of carrying out the calculation along this line may not pay off for larger $L-R$, as demonstrated by the next result.
\begin{theorem}\label{thm:bar_asymp}
  Suppose $(l,r,l',r')=(t+m,t,t+m,t)$ and let $m$ be fixed.
  Then if $m\equiv1\mod4$, $\oname{Re}H(\pi)>0$ for sufficiently large $t$; and if $m\equiv3\mod4$, $\oname{Re}H(\pi)<0$ for sufficiently large $t$.
\end{theorem}

\subsection{Proof of~\cref{thm:bar_sqrt}}
\label{sec:bar_sqrt}
Consider an arbitrary trigonometric polynomial:
\begin{equation}\label{eq:bar_trigpoly}
  C(\theta) = \sum_{k=0}^nc_k\cos(k\theta)\;,\quad c_k\in\mathbb{R}\;,
\end{equation}
then a simple bound on the coefficients are given in~\cref{lm:bar_trigpoly_bound} (see also~\cite{AGluchoff:1998a}).
\begin{lemma}\label{lm:bar_trigpoly_bound}
  If~\cref{eq:bar_trigpoly} is non-negative for $0\le\theta\le\pi$, then $\abs{c_k}<2c_0$ for all $1\le k\le m$.
\end{lemma}
\begin{proof}
  By orthogonality one has:
  \begin{displaymath}
    c_0 = \frac{1}{\pi}\int_0^{\pi}C(\theta)d\theta\;;\quad
    c_k = \frac{2}{\pi}\int_0^{\pi}C(\theta)\cos(k\theta)d\theta\;,\ 1\le k\le m\;.
  \end{displaymath}
  Thus if $C(\theta)\ge0$, one has $c_0\ge0$ and for all $1\le k\le m$:
  \begin{displaymath}
    \abs{c_k}\le\frac{2}{\pi}\int_0^\pi\abs{C(\theta)}\abs{\cos(k\theta)}d\theta<\frac{2}{\pi}\int_0^{\pi}C(\theta)d\theta = 2c_0\;.
  \end{displaymath}
  Note that the strict inequality is due to the fact that $C(\theta)$ is continuous and equality requires $C(\theta)=0$ for all $\cos(k\theta)\ne\pm1$.
\end{proof}

Writing:
\begin{equation}\label{eq:bar_sqrt_reH}
  \oname{Re}H(\theta) =  \sum_{k=0}^{l'}h_k\cos(k\theta) = \beta_0+\sum_{k=1}^{r'}(\beta_k+\beta_{-k})\cos(k\theta)+\sum_{k=r'+1}^{l'}\beta_{-k}\cos(k\theta)\;,
\end{equation}
where the $\beta$-coefficients are given by~\cref{eq:prelim_coef_node_zero} and~\cref{eq:prelim_coef_node_nz}.
The idea is to show that if $L-R$ is sufficiently large, there will be some $k\ge1$ such that $\abs{h_k}\ge 2h_0$, hence by~\cref{lm:bar_trigpoly_bound} $\oname{Re}H(\theta)$ cannot be non-negative.

For this purpose, we shall distinguish four different cases: (1) $L=2t+2m, R=2t$, (2) $L=2t+2m+1, R=2t$, (3) $L=2t+2m+1, R=2t+1$, and (4) $L=2t+2m+2, R=2t+1$, where $t$ and $m$ are non-negative integers such that $L>R$.
And we assume that $\oname{Re}H(\theta)\ge0$ for all $0\le\theta\le\pi$.

\medskip

\noindent
{\bf Case 1: $L=2t+2m$, $R=2t$, where $t\ge0$ and $m\ge1$}.

\smallskip

In this case, the four-index stencil is $(l,r,l',r')=(t+m,t,t+m,t)$ and one has:
\begin{align*}
  h_0 &= 4\zeta_0^{t+m,t} = 4\left(\frac{1}{t+1}+\cdots+\frac{1}{t+m}\right)\;; \\
  h_k &= \frac{2}{k}\left[\left(\frac{(t+m)!t!}{(t+m-k)!(t+k)!}\right)^2-\left(\frac{(t+m)!t!}{(t+m+k)!(t-k)!}\right)^2\right]\;,\quad1\le k\le t\;; \\
  h_k &= \frac{2}{k}\left(\frac{(t+m)!t!}{(t+m-k)!(t+k)!}\right)^2\;,\quad t+1\le k\le t+m\;.
\end{align*}
It is easy to see that all $h$'s are positive, and for all $1\le k\le t$ there is $(t+m)!t!<(t+m+k)!(t-k)!$; hence it is necessary for all $1\le k\le t+m$:
\begin{equation}\label{eq:bar_sqrt_case1_nec_raw}
  \frac{2}{k}\left[\left(\frac{(t+m)!t!}{(t+m-k)!(t+k)!}\right)^2-1\right] < h_k < 2h_0 = 8\left(\frac{1}{t+1}+\cdots+\frac{1}{t+m}\right)\;,
\end{equation}
or equivalently:
\begin{equation}\label{eq:bar_sqrt_case1_nec}
  2\left[\ln\left(1+\frac{m-k}{t+1}\right)+\cdots\ln\left(1+\frac{m-k}{t+k}\right)\right] < \ln\left(1+\frac{4k}{t+1}+\cdots+\frac{4k}{t+m}\right)\;.
\end{equation}
A bound on $m$ can already be obtained by setting $k=1$, in which case~\cref{eq:bar_sqrt_case1_nec} gives:
\begin{align}
  \notag
  &\ 2\ln\left(1+\frac{m-1}{t+1}\right)<\ln\left(1+\frac{4}{t+1}+\cdots+\frac{4}{t+m}\right)<\ln\left(1+\frac{4m}{t+1}\right) \\
  \label{eq:bar_sqrt_case1_linb}
  \Rightarrow&\ \frac{(m-1)^2}{2(m+1)}<t+1 \Leftrightarrow m-3+\frac{4}{m+1}<2t+2 \Rightarrow m\le 2t+4\;.
\end{align}
Note that the bound can be improved to $m\le 2t+3$ using the fact that $1/x$ is convex, hence $\frac{4}{t+1}+\cdots+\frac{4}{t+m}\le\frac{2m}{t+1}+\frac{2m}{t+m}$ and one must have:
\begin{displaymath}
  \left(1+\frac{m-1}{t+1}\right)^2<1+\frac{2m}{t+1}+\frac{2m}{t+m}\;;
\end{displaymath}
letting $m=2t+4$ one has $\frac{1}{t+m}<\frac{1}{3t+3}$ and the preceding inequality gives:
\begin{displaymath}
  \left(3+\frac{1}{t+1}\right)^2<1+\frac{2(2t+4)}{t+1}+\frac{2(2t+4)}{3t+3} = \frac{19}{3}+\frac{16/3}{t+1} < 7+\frac{6}{t+1}\;, 
\end{displaymath}
which is clearly a contradiction.
Thus one has $m\le2t+3$.

To get a better bound, we want to pick a $k$ such that the left hand side of~\cref{eq:bar_sqrt_case1_nec_raw} is approximately the largest.
For this purpose, we pick $k=\floor{m/2}$ and have $k\le m-k$, hence:
\begin{align}
  \notag
  &\ 2\left[\ln\left(1+\frac{k}{t+1}\right)+\cdots\ln\left(1+\frac{k}{t+k}\right)\right] \le 2\left[\ln\left(1+\frac{m-k}{t+1}\right)+\cdots\ln\left(1+\frac{m-k}{t+k}\right)\right] \\
  \label{eq:bar_sqrt_case1_halfm}
  <&\ \ln\left(1+\frac{4k}{t+1}+\cdots+\frac{4k}{t+m}\right)
  \le \ln\left(1+\frac{4k}{t+1}+\cdots+\frac{4k}{t+2k+1}\right)\;.
\end{align}
As $\ln(1+x)$ is concave, the left hand side satisfies:
\begin{displaymath}
  2\left[\ln\left(1+\frac{k}{t+1}\right)+\cdots\ln\left(1+\frac{k}{t+k}\right)\right] \ge k\ln\left(1+\frac{k}{t+1}\right)+k\ln\left(1+\frac{k}{t+k}\right)\;;
\end{displaymath}
and the right hand side clearly has the estimate:
\begin{displaymath}
  \frac{4k}{t+1}+\cdots+\frac{4k}{t+2k+1} \le (4k^2+2k)\left(\frac{1}{t+1}+\frac{1}{t+k}\right)\;.
\end{displaymath}
Thus~\cref{eq:bar_sqrt_case1_halfm} leads to:
\begin{equation}\label{eq:bar_sqrt_case1_halfm_simp}
  \left(1+\frac{k}{t+1}\right)^k\left(1+\frac{k}{t+k}\right)^k < 1 + \frac{4k^2+2k}{t+1} + \frac{4k^2+2k}{t+k}\;.
\end{equation}

Because $m\le2t+3$, there is $k\le t+1$. 
It is elemetary to show if $x\le1$, $(1+x)^{\frac{1}{x}}\ge2$; setting $x=\frac{k}{t+1}$ and $x=\frac{k}{t+k}$ one gets respectively $\left(1+\frac{k}{t+1}\right)^k \ge 2^{\frac{k^2}{t+1}}$ and $\left(1+\frac{k}{t+k}\right)^k \ge 2^{\frac{k^2}{t+k}}$.
Denoting $\gamma=\frac{k^2}{t+1}+\frac{k^2}{t+k}$, one thus has from~\cref{eq:bar_sqrt_case1_halfm_simp}:
\begin{displaymath}
  2^{\gamma} \le \left(1+\frac{k}{t+1}\right)^k\left(1+\frac{k}{t+k}\right)^k \le 1+\frac{4k^2+2k}{t+1}+\frac{4k^2+2k}{t+k} \le 1+6\gamma\;.
\end{displaymath}
It follows that $\gamma<5$ and:
\begin{displaymath}
  \frac{2k^2}{t+k}\le\gamma<5\quad\Rightarrow\quad k<\frac{5}{4}+\sqrt{\frac{5}{2}t+\frac{25}{16}}\;.
\end{displaymath}
Using $L-R=2m\le4k+2$ and $R=2t$, the preceding inequality gives:
\begin{displaymath}
  L-R \le 7+\sqrt{20R+25} < 7+\sqrt{21R+49}\;.
\end{displaymath}
Combining with $L-R=2m\le4t+6=2R+6$, we obtain $L-R\le\min(2R+6,9+\sqrt{21R+49})$.

\medskip

The other three cases are handled similarly, and we omit the details here and leave them in~\cref{app:sqrt}.
To this end, we finish the proof of~\cref{thm:bar_sqrt}.

\subsection{Combinatorics and ReH}
\label{sec:bar_reh}
In this section we establish some combinatorics tools that help with proving~\cref{thm:bar_pi} and~\cref{thm:bar_asymp}.
Let us begin with the simpler case $(L,R)=(2t+2m,2t)$ or $(l,r,l',r')=(t+m,t,t+m,t)$ as before and have:
\begin{align}
  \notag
  \oname{Re}H(\theta) =&\ 4\zeta_0^{t+m,t}+\sum_{k=1}^t\frac{2}{k}\left[\left(\frac{(t+m)!t!}{(t+m-k)!(t+k)!}\right)^2-\left(\frac{(t+m)!t!}{(t+m+k)!(t-k)!}\right)^2\right]\cos(k\theta) \\
  \label{eq:bar_reh_sym}
  &\ +\sum_{k=t+1}^{t+m}\frac{2}{k}\left(\frac{(t+m)!t!}{(t+m-k)!(t+k)!}\right)^2\cos(k\theta)\;.
\end{align}
The key is to define the polynomial $P_{t,m}(x)=(x+t+1)^{(m)}$, where $(x)^{(m)}=x(x+1)\cdots(x+m-1)$ is the rising factorial, then for all $1\le k\le t$:
\begin{displaymath}
  \left(\frac{(t+m)!t!}{(t+m-k)!(t+k)!}\right)^2-\left(\frac{(t+m)!t!}{(t+m+k)!(t-k)!}\right)^2 
  = \left(\frac{(t+m)!t!}{(t+m+k)!(t+m-k)!}\right)^2Q_{t,m}(k)\;,
\end{displaymath}
where $Q_{t,m}(x)\eqdef P^2_{t,m}(x)-P^2_{t,m}(-x)$.
Noting that $P_{t,m}(t+j)=0$, $1\le j\le m$, one immediately obtains for $t+1\le k\le t+m$:
\begin{displaymath}
  \left(\frac{(t+m)!t!}{(t+m-k)!(t+k)!}\right)^2 = \left(\frac{(t+m)!t!}{(t+m+k)!(t+m-k)!}\right)^2P_{t,m}^2(k) = \left(\frac{(t+m)!t!}{(t+m+k)!(t+m-k)!}\right)^2Q_{t,m}(k)\;. 
\end{displaymath}
Thus~\cref{eq:bar_reh_sym} can be re-written as:
\begin{displaymath}
  \oname{Re}H(\theta) = 4\zeta_0^{t+m,t}+\sum_{k=1}^{t+m}\frac{2Q_{t,m}(k)}{k}\left(\frac{(t+m)!t!}{(t+m-k)!(t+m+k)!}\right)^2\cos(k\theta)\;.
\end{displaymath}

To move on, we introduce the following notation: Let $P(x)$ be a polynomial in $x$, then we write:
\begin{equation}\label{eq:bar_reh_poly}
  P(x) = [P]_0 + [P]_1x + \cdots + [P]_nx^n\;;
\end{equation}
that is, $[P]_k$ is the coefficient of $x^k$, $k\ge0$ in $P$ and $n$ is the degree.
It is clear that one can define $[P]_k=0$ for all $k>n$ so that the notation holds for all $k\ge0$.
For example:
\begin{equation}\label{eq:bar_reh_p_coef}
  [P_{t,m}]_0 = (t+1)\cdots(t+m) = \frac{(t+m)!}{t!}\;,\quad
  [P_{t,m}]_1 = \frac{(t+m)!}{t!}\zeta_0^{t+m,t}\;.
\end{equation}
Noting that:
\begin{displaymath}
  \frac{Q_{t,m}(x)}{x} = \frac{P_{t,m}(x)-P_{t,m}(-x)}{x}\cdot\left[P_{t,m}(x)+P_{t,m}(-x)\right]
\end{displaymath}
one has $Q_{t,m}(x)/x$ is an even polynomial with constant term:
\begin{displaymath}
  \left[\frac{Q_{t,m}(x)}{x}\right]_0 = 4[P_{t,m}]_1[P_{t,m}]_0 = 4\zeta_0^{t+m,t}\left(\frac{(t+m)!}{t!}\right)^2\;.
\end{displaymath}
Let us write:
\begin{equation}\label{eq:bar_reh_q}
  \frac{Q_{t,m}(x)}{x} = \left(\frac{(t+m)!}{t!}\right)^2\left[q_{t,m;0}+q_{t,m;1}x^2+\cdots+q_{t,m;m-1}x^{2(m-1)}\right]\;,
\end{equation}
then $q_{t,m;0}=4\zeta_0^{t+m,t}$ and one has:
\begin{displaymath}
  \oname{Re}H(\theta) = q_{t,m;0}+2\sum_{k=1}^{t+m}\left(\sum_{j=0}^{m-1}q_{t,m;j}k^{2j}\right)\left(\frac{(t+m)!(t+m)!}{(t+m-k)!(t+m+k)!}\right)^2\cos(k\theta)\;.
\end{displaymath}
For convenience, we define a function $C_{m,n}(\theta)$ with two indices as follows:
\begin{equation}\label{eq:bar_reh_cfun}
  C_{m,n}(\theta) = \sum_k\binom{m}{n+k}\binom{m}{n-k}\cos(k\theta)\;,
\end{equation}
where the summation is taken over all integers and we adopt the convention that $\binom{m}{n}=0$ for all $n<0$ or $n>m$.
Then the real part of $H(\theta)$ has the following formulation:
\begin{equation}\label{eq:bar_reh_crep}
  \oname{Re}H(\theta) = \binom{2t+2m}{t+m}^{-2}\sum_{j=0}^{m-1}(-1)^jq_{t,m;j}\frac{d^{2j}}{d\theta^{2j}}C_{2t+2m,t+m}(\theta)\;.
\end{equation}
This equation plays a critical role in the proof of~\cref{thm:bar_pi}, for which we'll need the following results that are proved in~\cref{app:comb}.
\begin{lemma}\label{lm:bar_reh_cfun_alt}
  For all $m\ge n\ge0$, there is:
  \begin{equation}\label{eq:bar_reh_cfun_alt}
    C_{m,n}(\theta) = \sum_k\binom{m}{m-k}\binom{2k}{2n}\sin^{2(m-k)}\frac{\theta}{2}\cos^{2(k-n)}\frac{\theta}{2}\;.
  \end{equation}
\end{lemma}
\begin{lemma}\label{lm:bar_reh_cfun_der}
  For all $m,n\ge1$, there is:
  \begin{equation}\label{eq:bar_reh_cfun_der}
    C''_{m,n}(\theta) = m^2C_{m-1,n-1}(\theta)-n^2C_{m,n}(\theta)\;.
  \end{equation}
\end{lemma}
Note that a direct consequence of~\cref{lm:bar_reh_cfun_alt} is:
\begin{equation}\label{eq:bar_reh_cfun_pi}
  C_{m,n}(\pi) = \binom{m}{n}\;.
\end{equation}
\begin{remark}
  When $m=1$, \cref{eq:bar_reh_crep} reduces to $\oname{Re}H(\theta) = 4\zeta_0^{t+1,t}\binom{2t+2}{t+1}^{-2}C_{2t+2,t+1}(\theta)$, which is non-negative by~\cref{lm:bar_reh_cfun_alt}.
  This actually verifies~\cref{eq:prelim_stab_cond_h} for methods with stencil $(l,r,l',r')=(t+1,t,t+1,t)$.
  It is not a stability proof, though, as we are still lacking~\cref{eq:prelim_stab_cond_g} that is much more complicated.
\end{remark}

\bigskip

In the rest of this section we show that the same strategy can be used to study the other stencils.
When $(L,R)=(2t+2m+1,2t+1)$ or $(l,r,l',r')=(t+m+1,t+1,t+m,t)$, one has:
\begin{displaymath}
  \oname{Re}H(\theta) = 2(\zeta_0^{t+m+1,t+1}+\zeta_0^{t+m,t}) + \sum_{k=1}^{t+m}\frac{\binom{2t+2m+1}{t+m+k}\binom{2t+2m+1}{t+m-k}}{\binom{2t+2m+1}{t+m}^2}\frac{(t+1)!t!}{(t+m+1)!(t+m)!}\frac{2\check{Q}_{t,m}(k)}{k}\cos(k\theta)\;,
\end{displaymath}
where $\check{Q}_{t,m}(x)\eqdef P_{t+1,m}(x)P_{t,m}(x)-P_{t+1,m}(-x)P_{t,m}(-x)$ and $\check{Q}_{t,m}(x)/x$ is an even polynomial of degree $2(m-1)$, whose constant term is given by:
\begin{displaymath}
  \left[\frac{\check{Q}_{t,m}(x)}{x}\right]_0 
  = 2[P_{t+1,m}]_0[P_{t,m}]_1+2[P_{t+1,m}]_1[P_{t,m}]_0
  = \frac{2(t+m+1)!(t+m)!}{(t+1)!t!}(\zeta_0^{t+m,t}+\zeta_0^{t+m+1,t+1})\;.
\end{displaymath}
Writing:
\begin{equation}\label{eq:bar_reh_q_check}
  \frac{\check{Q}_{t,m}(x)}{x} = \frac{(t+m+1)!(t+m)!}{(t+1)!t!}\left[\check{q}_{t,m;0}+\check{q}_{t,m;1}x^2+\cdots+\check{q}_{t,m;m-1}x^{2(m-1)}\right]\;,
\end{equation}
where $\check{q}_{t,m;0}=2(\zeta_0^{t+m+1,t+1}+\zeta_0^{t+m,t})$, one has the next representation of $\oname{Re}H(\theta)$:
\begin{equation}\label{eq:bar_reh_crep_check}
  \oname{Re}H(\theta) = \binom{2t+2m+1}{t+m}^{-2}\sum_{j=0}^{m-1}(-1)^j\check{q}_{t,m;j}\frac{d^{2j}}{d\theta^{2j}}C_{2t+2m+1,t+m}(\theta)\;.
\end{equation}

\medskip

Next consider $(L,R)=(2t+2m+1,2t)$ or $(l,r,l',r')=(t+m+1,t,t+m,t)$, then:
\begin{displaymath}
  \oname{Re}H(\theta) = 2(\zeta_0^{t+m+1,t+1}+\zeta_0^{t+m,t}) + \sum_{k=1}^{t+m}\frac{\binom{2t+2m+1}{t+m+k}\binom{2t+2m+1}{t+m-k}}{\binom{2t+2m+1}{t+m}^2}\frac{t!t!}{(t+m+1)!(t+m)!}\frac{2\hat{Q}_{t,m}(k)}{k}\cos(k\theta)\;,
\end{displaymath}
where $\hat{Q}_{t,m}(x)=P_{t,m+1}(x)P_{t,m}(x)-P_{t,m+1}(-x)P_{t,m}(-x)$.
It follows that $\hat{Q}_{t,m}(x)/x$ is an even polynomial of degree $2m$ with constant term $\frac{2(t+m+1)!(t+m)!}{t!t!}(\zeta_0^{t+m+1,t}+\zeta_0^{t+m,t})$.
Writing:
\begin{equation}\label{eq:bar_reh_q_hat}
  \frac{\hat{Q}_{t,m}(x)}{x} = \frac{(t+m+1)!(t+m)!}{t!t!}\left[\hat{q}_{t,m;0}+\hat{q}_{t,m;1}x^2+\cdots+\hat{q}_{t,m;m}x^{2m}\right]\;,
\end{equation}
we obtain the following formula of $\oname{Re}H(\theta)$:
\begin{equation}\label{eq:bar_reh_crep_hat}
  \oname{Re}H(\theta) = \binom{2t+2m+1}{t+m}^{-2}\sum_{j=0}^{m}(-1)^j\hat{q}_{t,m;j}\frac{d^{2j}}{d\theta^{2j}}C_{2t+2m+1,t+m}(\theta)\;.
\end{equation}

\medskip

Lastly consider $(L,R)=(2t+2m+2,2t)$ or $(l,r,l',r')=(t+m+2,t+1,t+m,t)$, then:
\begin{displaymath}
  \oname{Re}H(\theta) = 2(\zeta_0^{t+m+2,t+1}+\zeta_0^{t+m,t}) + \sum_{k=1}^{t+m}\frac{\binom{2t+2m+2}{t+m+k}\binom{2t+2m+2}{t+m-k}}{\binom{2t+2m+2}{t+m}^2}\frac{(t+1)!t!}{(t+m+2)!(t+m)!}\frac{2\tilde{Q}_{t,m}(k)}{k}\cos(k\theta)\;,
\end{displaymath}
where $\tilde{Q}_{t,m}(x)=P_{t+1,m+1}(x)P_{t,m}(x)-P_{t+1,m+1}(-x)P_{t,m}(-x)$.
Then $\tilde{Q}_{t,m}(x)/x$ is an even polynomial of degree $2m$ with constant term $\frac{2(t+m+2)!(t+m)!}{(t+1)!t!}(\zeta_0^{t+m+2,t+1}+\zeta_0^{t+m,t})$.
Defining
\begin{equation}\label{eq:bar_reh_q_tilde}
  \frac{\tilde{Q}_{t,m}(x)}{x} = \frac{(t+m+2)!(t+m)!}{(t+1)!t!}\left[\tilde{q}_{t,m;0}+\tilde{q}_{t,m;1}x^2+\cdots+\tilde{q}_{t,m;m}x^{2m}\right]\;,
\end{equation}
we have the next form of $\oname{Re}H(\theta)$:
\begin{equation}\label{eq:bar_reh_crep_tilde}
  \oname{Re}H(\theta) = \binom{2t+2m+2}{t+m}^{-2}\sum_{j=0}^{m}(-1)^j\tilde{q}_{t,m;j}\frac{d^{2j}}{d\theta^{2j}}C_{2t+2m+2,t+m}(\theta)\;.
\end{equation}


\subsection{Proof of~\cref{thm:bar_pi}}
\label{sec:bar_pi}
The proof of instability of different HV schemes is organized in a way such that those using the same representation of $\oname{Re}H$ are studied together, as summarized in~\cref{tb:bar_pi_comb}.
\begin{table}\centering
  \caption{The order in which the HV schemes whose instability is proved in~\cref{sec:bar_pi} and related representation of $\oname{Re}H$ that is used.}
  \label{tb:bar_pi_comb}
  \begin{tabular}{@{}lclll@{}}
    \toprule[.5mm]
    $\oname{Re}H$ & & \multicolumn{3}{l}{$L-R$ and $(l,r,l',r')$} \\ \cmidrule[.2mm]{1-5}
    \cref{eq:bar_reh_crep}       & & $4$,~~$(t+2,t,t+2,t)$;   & & $6$,~~$(t+3,t,t+3,t)$.   \\
    \cref{eq:bar_reh_crep_check} & & $4$,~~$(t+3,t+1,t+2,t)$; & & $6$,~~$(t+4,t+1,t+3,t)$. \\
    \cref{eq:bar_reh_crep_hat}   & & $5$,~~$(t+3,t,t+2,t)$;   & & $7$,~~$(t+4,t,t+3,t)$.   \\
    \cref{eq:bar_reh_crep_tilde} & & $5$,~~$(t+4,t+1,t+2,t)$; & & $7$,~~$(t+5,t+1,t+3,t)$. \\
    \bottomrule[.4mm]
  \end{tabular}
\end{table}

\medskip

\noindent
{\it 1. $L-R=4$ and $(l,r,l',r')=(t+2,t,t+2,t)$}.
Setting $m=2$ in~\cref{eq:bar_reh_crep}, one obtains:
\begin{align*}
  \binom{2t+4}{t+2}^2\oname{Re}H(\theta) &= q_{t,2;0}C_{2t+4,t+2}(\theta) - q_{t,2;1}C''_{2t+4,t+2}(\theta) \\
  &= q_{t,2;0}C_{2t+4,t+2}(\theta)-q_{t,2;1}\left[(2t+4)^2C_{2t+3,t+1}(\theta)-(t+2)^2C_{2t+4,t+2}(\theta)\right]\;,
\end{align*}
where~\cref{lm:bar_reh_cfun_der} is used in the last equality.
Recall $P_{t,2}(x) = (x+t+1)(x+t+2)$, one has:
\begin{displaymath}
  \frac{Q_{t,2}(x)}{x} = \frac{(x+t+1)^2(x+t+2)^2-(-x+t+1)^2(-x+t+2)^2}{x} = 4(2t+3)[x^2+(t+1)(t+2)]\;,
\end{displaymath}
and thus:
\begin{displaymath}
  q_{t,2;0} = 4\zeta_0^{t+2,t} = \frac{4(2t+3)}{(t+1)(t+2)}\;,\quad
  q_{t,2;1} = \frac{4(2t+3)}{(t+1)^2(t+2)^2} = \frac{q_{t,2;0}}{(t+1)(t+2)}\;.
\end{displaymath}
Using~\cref{eq:bar_reh_cfun_pi}, we obtain:
\begin{displaymath}
  \frac{\binom{2t+4}{t+2}^2\oname{Re}H(\pi)}{q_{t,2;0}} = \binom{2t+4}{t+2}-\frac{1}{(t+1)(t+2)}\left[(2t+4)^2\binom{2t+3}{t+1}-(t+2)^2\binom{2t+4}{t+2}\right] = -\frac{\binom{2t+4}{t+2}}{t+1} < 0\;.
\end{displaymath}
Hence $\oname{Re}H(\pi)<0$ and by~\cref{lm:prelim_stab}, the scheme is unstable. 

\medskip

\noindent
{\it 2. $L-R=6$ and $(l,r,l',r')=(t+3,t,t+3,t)$}.
Similar as before, setting $m=3$ in~\cref{eq:bar_reh_crep} gives:
\begin{displaymath}
  \binom{2t+6}{t+3}^2\oname{Re}H(\theta) = q_{t,3;0}C_{2t+6,t+3}(\theta) - q_{t,3;1}C''_{2t+6,t+3}(\theta) + q_{t,3;2}C''''_{2t+6,t+3}(\theta)\;.
\end{displaymath}
For the $q$-coefficients, using $P_{t,3}(x)=(x+t+1)(x+t+2)(x+t+3)$ one gets:
\begin{displaymath}
  q_{t,3;0} = 4\zeta_0^{t+3,t}\;,\quad
  q_{t,3;1} = \frac{4+12(t+2)\zeta_0^{t+3,t}}{(t+1)(t+2)(t+3)}\;,\quad
  q_{t,3;2} = \frac{12}{(t+1)^2(t+2)(t+3)^2}\;;
\end{displaymath}
and for the $C$-functions, using~\cref{eq:bar_reh_cfun_pi} and~\cref{lm:bar_reh_cfun_der} there is:
\begin{align*}
  &C_{2t+6,t+3}(\pi) = \binom{2t+6}{t+3}\;,\quad
  C''_{2t+6,t+3}(\pi) = (2t+6)^2\binom{2t+5}{t+2}-(t+3)^2\binom{2t+6}{t+3}\;,\\
  &C''''_{2t+6,t+3}(\pi) = (2t+6)^2(2t+5)^2\binom{2t+4}{t+1}-(2t+6)^2[2t^2+10t+13]\binom{2t+5}{t+2}+(t+3)^4\binom{2t+6}{t+3}\;.
\end{align*}
Putting everything together, we obtain:
\begin{displaymath}
  \binom{2t+6}{t+3}^2\oname{Re}H(\pi) = -\frac{8}{t+1}\left(2+\frac{3}{t+1}+\frac{1}{t+2}+\frac{1}{t+3}\right)\binom{2t+6}{t+3} < 0\;,
\end{displaymath}
which indicates the instability of the method.

\medskip

\noindent
{\it 3. $L-R=4$ and $(l,r,l',r')=(t+3,t+1,t+2,t)$}.
Setting $m=2$ in~\cref{eq:bar_reh_crep_check} gives:
\begin{align*}
  \binom{2t+5}{t+2}^2\oname{Re}H(\theta) &= \check{q}_{t,2;0}C_{2t+5,t+2}(\theta) - \check{q}_{t,2;1}C''_{2t+5,t+2}(\theta) \\
  &= \check{q}_{t,2;0}C_{2t+5,t+2}(\theta)-\check{q}_{t,2;1}\left[(2t+5)^2C_{2t+4,t+1}(\theta)-(t+2)^2C_{2t+5,t+2}(\theta)\right]\;.
\end{align*}
Using $\check{Q}_{t,2}(x) = (x+t+1)(x+t+2)^2(x+t+3) - (t+1-x)(t+2-x)^2(t+3-x) = (t+1)(t+2)^2(t+3)x(\check{q}_{t,2;0}+\check{q}_{t,2;1}x^2)$ one computes:
\begin{displaymath}
  \check{q}_{t,2;0} = 2(\zeta_0^{t+3,t+1}+\zeta_0^{t+2,t}) = \frac{4(2t^2+8t+7)}{(t+1)(t+2)(t+3)}\;,\quad
  \check{q}_{t,2;1} = \frac{8}{(t+1)(t+2)(t+3)}\;.
\end{displaymath}
Thus by~\cref{eq:bar_reh_cfun_pi} we compute:
\begin{displaymath}
  \binom{2t+5}{t+2}^2\oname{Re}H(\pi) = -\frac{4(2t+5)}{(t+1)(t+2)(t+3)}\binom{2t+5}{t+2} < 0\;.
\end{displaymath}

\medskip

\noindent
{\it 4. $L-R=6$ and $(l,r,l',r')=(t+4,t+1,t+3,t)$}.
Setting $m=3$ in~\cref{eq:bar_reh_crep_check} gives:
\begin{align*}
  &\binom{2t+7}{t+3}^2\oname{Re}H(\pi) = \check{q}_{t,3;0}C_{2t+7,t+3}(\pi) - \check{q}_{t,3;1}C''_{2t+7,t+3}(\pi) + \check{q}_{t,3;2}C''''_{2t+7,t+3}(\pi) \\
  &= \left(\check{q}_{t,3;0}+(t+3)^2\check{q}_{t,3;1}+(t+3)^4\check{q}_{t,3;2}\right)\binom{2t+7}{t+3} - \\
  &\quad (2t+7)^2\left(\check{q}_{t,3;1}+(2t^2+10t+13)\check{q}_{t,3;2}\right)\binom{2t+6}{t+2} + (2t+6)^2(2t+7)^2\check{q}_{t,3;2}\binom{2t+5}{t+1}\;.
\end{align*}
Using $\check{Q}_{t,3}(x) = (x+t+1)(x+t+2)^2(x+t+3)^2(x+t+4) - (t+1-x)(t+2-x)^2(t+3-x)^2(t+4-x) = (t+1)(t+2)^2(t+3)^2(t+4)x(\check{q}_{t,3;0}+\check{q}_{t,3;1}x^2+\check{q}_{t,3;2}x^4)$ one computes:
\begin{align*}
  \check{q}_{t,3;0} &= 2(\zeta_0^{t+4,t+1}+\zeta_0^{t+3,t}) = \frac{2(2t+5)(3t^2+15t+14)}{(t+1)(t+2)(t+3)(t+4)}\;, \\
  \check{q}_{t,3;1} &= \frac{6(2t+5)(3t^2+15t+17)}{(t+1)(t+2)^2(t+3)^2(t+4)}\;,\quad  \check{q}_{t,3;2} = \frac{6(2t+5)}{(t+1)(t+2)^2(t+3)^2(t+4)}\;.
\end{align*}
Substituting them into the previous equation gives:
\begin{displaymath}
  \binom{2t+7}{t+3}^2\oname{Re}H(\pi) = -\frac{2(2t+5)(3t^2+27t+58)}{(t+1)(t+2)(t+3)(t+4)}\binom{2t+7}{t+3} < 0\;.
\end{displaymath}

\medskip

\noindent
{\it 5. $L-R=5$ and $(l,r,l',r')=(t+3,t,t+2,t)$}.
Setting $m=2$ in~\cref{eq:bar_reh_crep_hat} gives:
\begin{align*}
  &\binom{2t+5}{t+2}^2\oname{Re}H(\pi) = \hat{q}_{t,2;0}C_{2t+5,t+2}(\pi) - \hat{q}_{t,2;1}C''_{2t+5,t+2}(\pi) + \hat{q}_{t,2;2}C''''_{2t+5,t+2}(\pi) \\
  &= \left(\hat{q}_{t,2;0}+(t+2)^2\hat{q}_{t,2;1}+(t+2)^4\hat{q}_{t,2;2}\right)\binom{2t+5}{t+2} - \\
  &\quad (2t+5)^2\left(\hat{q}_{t,2;1}+(2t^2+6t+5)\hat{q}_{t,2;2}\right)\binom{2t+4}{t+1} + (2t+4)^2(2t+5)^2\hat{q}_{t,2;2}\binom{2t+3}{t}\;.
\end{align*}
Using $\hat{Q}_{t,2}(x) = (x+t+1)^2(x+t+2)^2(x+t+3)-(t+1-x)^2(t+2-x)^2(t+3-x) = (t+1)^2(t+2)^2(t+3)(\hat{q}_{t,2;0}+\hat{q}_{t,2;1}x^2+\hat{q}_{t,2;2}x^4)$, there is:
\begin{align*}
  \hat{q}_{t,2;0} &= 2(\zeta_0^{t+3,t}+\zeta_0^{t+2,t}) = \frac{2(5t^2+21t+20)}{(t+1)(t+2)(t+3)}\;, \\
  \hat{q}_{t,2;1} &= \frac{2(10t^2+36t+31)}{(t+1)^2(t+2)^2(t+3)}\;,\quad
  \hat{q}_{t,2;2} = \frac{2}{(t+1)^2(t+2)^2(t+3)}\;.
\end{align*}
It follows that:
\begin{displaymath}
  \binom{2t+5}{t+2}^2\oname{Re}H(\pi) = -\frac{4(2t+5)(t^2+6t+7)}{(t+1)^2(t+2)(t+3)}\binom{2t+5}{t+2} < 0\;.
\end{displaymath}

\medskip

\noindent
{\it 6. $L-R=7$ and $(l,r,l',r')=(t+4,t,t+3,t)$}.
Setting $m=3$ in~\cref{eq:bar_reh_crep_hat} gives:
\begin{align*}
  &\binom{2t+7}{t+3}^2\oname{Re}H(\pi) = \hat{q}_{t,3;0}C_{2t+7,t+3}(\pi) - \hat{q}_{t,3;1}C''_{2t+7,t+3}(\pi) + \hat{q}_{t,3;2}C''''_{2t+7,t+3}(\pi) - \hat{q}_{t,3;3}C^{''''''}_{2t+7,t+3}(\pi) \\
  &= \left(\hat{q}_{t,3;0}+(t+3)^2\hat{q}_{t,3;1}+(t+3)^4\hat{q}_{t,3;2}+(t+3)^6\right)\binom{2t+7}{t+3} - \\
  &\quad (2t+7)^2\left(\hat{q}_{t,3;1}+(2t^2+10t+13)\hat{q}_{t,3;2}+(3t^4+30t^3+115t^2+200t+133)\hat{q}_{t,3;3}\right)\binom{2t+6}{t+2} + \\
  &\quad (2t+6)^2(2t+7)^2(\hat{q}_{t,3;2}+(3t^2+12t+14)\hat{q}_{t,3;3})\binom{2t+5}{t+1} - (2t+5)^2(2t+6)^2(2t+7)^2\hat{q}_{t,3;3}\binom{2t+4}{t}\;.
\end{align*}
Using $\hat{Q}_{t,3}(x) = (x+t+1)^2(x+t+2)^2(x+t+3)^2(x+t+4)-(t+1-x)^2(t+2-x)^2(t+3-x)^2(t+4-x) = (t+1)^2(t+2)^2(t+3)^2(t+4)(\hat{q}_{t,3;0}+\hat{q}_{t,3;1}x^2+\hat{q}_{t,3;2}x^4+\hat{q}_{t,3;3}x^6)$ one has:
\begin{align*}
  \hat{q}_{t,3;0} &= 2(\zeta_0^{t+4,t}+\zeta_0^{t+3,t}) = \frac{2(7t^3+54t^2+129t+94)}{(t+1)(t+2)(t+3)(t+4)}\;, \\
  \hat{q}_{t,3;1} &= \frac{2(35t^4+320t^3+1060t^2+1504t+769)}{(t+1)^2(t+2)^2(t+3)^2(t+4)}\;,\quad
  \hat{q}_{t,3;2} = \frac{2(21t^2+96t+106)}{(t+1)^2(t+2)^2(t+3)^2(t+4)}\;, \\
  \hat{q}_{t,3;3} &= \frac{2}{(t+1)^2(t+2)^2(t+3)^2(t+4)}\;.
\end{align*}
And finally:
\begin{displaymath}
  \binom{2t+7}{t+3}^2\oname{Re}H(\pi) = -\frac{4(2t+5)^2(2t+7)}{(t+1)(t+2)(t+3)(t+4)}\binom{2t+7}{t+3} < 0\;.
\end{displaymath}

\medskip

\noindent
{\it 7. $L-R=5$ and $(l,r,l',r')=(t+4,t+1,t+2,t)$}.
Setting $m=2$ in~\cref{eq:bar_reh_crep_tilde} gives:
\begin{align*}
  &\binom{2t+6}{t+2}^2\oname{Re}H(\pi) = \tilde{q}_{t,2;0}C_{2t+6,t+2}(\pi) - \tilde{q}_{t,2;1}C''_{2t+6,t+2}(\pi) + \tilde{q}_{t,2;2}C''''_{2t+6,t+2}(\pi) \\
  &= \left(\tilde{q}_{t,2;0}+(t+2)^2\tilde{q}_{t,2;1}+(t+2)^4\tilde{q}_{t,2;2}\right)\binom{2t+6}{t+2} - \\
  &\quad (2t+6)^2\left(\tilde{q}_{t,2;1}+(2t^2+6t+5)\tilde{q}_{t,2;2}\right)\binom{2t+5}{t+1} + (2t+5)^2(2t+6)^2\tilde{q}_{t,2;2}\binom{2t+4}{t}\;.
\end{align*}
Using $\tilde{Q}_{t,2}(x) = (x+t+1)(x+t+2)^2(x+t+3)(x+t+4)-(t+1-x)(t+2-x)^2(t+3-x)(t+4-x) = (t+1)(t+2)^2(t+3)(t+4)(\tilde{q}_{t,2;0}+\tilde{q}_{t,2;1}x^2+\tilde{q}_{t,2;2}x^4)$ one has:
\begin{align*}
  \tilde{q}_{t,2;0} &= 2(\zeta_0^{t+4,t+1}+\zeta_0^{t+2,t}) = \frac{2(5t^3+38t^2+89t+62)}{(t+1)(t+2)(t+3)(t+4)}\;, \\
  \tilde{q}_{t,2;1} &= \frac{2(10t^2+48t+55)}{(t+1)(t+2)^2(t+3)(t+4)}\;,\quad
  \tilde{q}_{t,2;2} = \frac{2}{(t+1)(t+2)^2(t+3)(t+4)}\;.
\end{align*}
Lastly we obtain:
\begin{displaymath}
  \binom{2t+6}{t+2}^2\oname{Re}H(\pi) = -\frac{4(2t+5)^2}{(t+1)(t+2)(t+4)}\binom{2t+6}{t+2} < 0\;.
\end{displaymath}

\medskip

\noindent
{\it 8. $L-R=7$ and $(l,r,l',r')=(t+5,t+1,t+3,t)$}.
Setting $m=3$ in~\cref{eq:bar_reh_crep_tilde} gives:
\begin{align*}
  &\binom{2t+8}{t+3}^2\oname{Re}H(\pi) = \tilde{q}_{t,3;0}C_{2t+8,t+3}(\pi) - \tilde{q}_{t,3;1}C''_{2t+8,t+3}(\pi) + \tilde{q}_{t,3;2}C''''_{2t+8,t+3}(\pi) - \tilde{q}_{t,3;3}C^{''''''}_{2t+8,t+3}(\pi) \\
  &= \left(\tilde{q}_{t,3;0}+(t+3)^2\tilde{q}_{t,3;1}+(t+3)^4\tilde{q}_{t,3;2}+(t+3)^6\right)\binom{2t+8}{t+3} - \\
  &\quad (2t+8)^2\left(\tilde{q}_{t,3;1}+(2t^2+10t+13)\tilde{q}_{t,3;2}+(3t^4+30t^3+115t^2+200t+133)\tilde{q}_{t,3;3}\right)\binom{2t+7}{t+2} + \\
  &\quad (2t+7)^2(2t+8)^2(\tilde{q}_{t,3;2}+(3t^2+12t+14)\tilde{q}_{t,3;3})\binom{2t+6}{t+1} - (2t+6)^2(2t+7)^2(2t+8)^2\tilde{q}_{t,3;3}\binom{2t+5}{t}\;.
\end{align*}
Using $\tilde{Q}_{t,3}(x) = (x+t+1)(x+t+2)^2(x+t+3)^2(x+t+4)(x+t+5)-(t+1-x)(t+2-x)^2(t+3-x)^2(t+4-x)(t+5-x) = (t+1)(t+2)^2(t+3)^2(t+4)(t+5)(\tilde{q}_{t,3;0}+\tilde{q}_{t,3;1}x^2+\tilde{q}_{t,3;2}x^4+\tilde{q}_{t,3;3}x^6)$:
\begin{align*}
  \tilde{q}_{t,3;0} &= 2(\zeta_0^{t+5,t+1}+\zeta_0^{t+3,t}) = \frac{2(7t^4+85t^3+363t^2+635t+374)}{(t+1)(t+2)(t+3)(t+4)(t+5)}\;, \\
  \tilde{q}_{t,3;1} &= \frac{2(35t^4+400t^3+1660t^2+2960t+1909)}{(t+1)(t+2)^2(t+3)^2(t+4)(t+5)}\;,\quad
  \tilde{q}_{t,3;2} = \frac{2(21t^2+120t+166)}{(t+1)(t+2)^2(t+3)^2(t+4)(t+5)}\;, \\
  \tilde{q}_{t,3;3} &= \frac{2}{(t+1)(t+2)^2(t+3)^2(t+4)(t+5)}\;.
\end{align*}
Lastly we compute:
\begin{displaymath}
  \binom{2t+8}{t+3}^2\oname{Re}H(\pi) = -\frac{8(2t+7)(t^2+8t+13)}{(t+1)(t+2)(t+3)(t+5)}\binom{2t+8}{t+3} < 0\;.
\end{displaymath}

\medskip

\noindent
Summarizing all cases, we complete the proof of~\cref{thm:bar_pi}.

\subsection{Proof of~\cref{thm:bar_asymp}}
\label{sec:bar_asymp}
Let $(l,r,l',r')=(t+m,t,t+m,t)$ where $m\ge1$, then $\oname{Re}H$ is given by~\cref{eq:bar_reh_crep}.
As we're interested in $\oname{Re}H(\pi)$ at fixed $m$ as $t\to\infty$, we shall expand each term of~\cref{eq:bar_reh_crep} in $1/t$.
Particularly, we shall need such estimates for $q_{t,m;j}$, $0\le j\le m-1$ and $\frac{d^{2j}}{d\theta^{2j}}C_{2t+2m,t+m}(\pi)$.
For notational simplicity, we denote $O(1/t)$ by $\varepsilon$.

\medskip

First we estimate coefficients $q_{t,m;j}$, defined for each $0\le j\le m-1$:
\begin{align*}
  q_{t,m;j} &= \left(\frac{t!}{(t+m)!}\right)^2[Q_{t,m}(x)]_{2j+1} = \left(\frac{t!}{(t+m)!}\right)^2[P_{t,m}^2(x)-P_{t,m}^2(-x)]_{2j+1} \\
  &= 2\left(\frac{t!}{(t+m)!}\right)^2\sum_{k=\max(0,2j+1-m)}^{\min(m,2j+1)}[P_{t,m}]_k[P_{t,m}]_{2j+1-k}\;.
\end{align*}
Because $P_{t,m}(x) = (x+t+1)(x+t+2)\cdots(x+t+m)$, there is:
\begin{displaymath}
  [P_{t,m}]_k = \frac{(t+m)!}{t!}\left[\binom{m}{k}\frac{1}{t^k}+O(\frac{1}{t^{k+1}})\right] = \binom{m}{k}t^{m-k}(1+\oone)\;,\ 0\le k\le m\;;
\end{displaymath}
therefore:
\begin{displaymath}
  [P_{t,m}]_k[P_{t,m}]_{2j+1-k} = \binom{m}{k}\binom{m}{2j+1-k}t^{2m-2j-1}(1+\oone)\;.
\end{displaymath}
Hence we obtain an estimate for $q_{t,m;j}$ as:
\begin{equation}\label{eq:bar_asymp_q}
  q_{t,m;j} = 2\left(\frac{t!}{(t+m)!}\right)^2\left[\sum_k\binom{m}{k}\binom{m}{2j+1-k}\right]t^{2m-2j-1}(1+\oone)
  = 2\binom{2m}{2j+1}t^{-2j-1}(1+\oone)\;.
\end{equation}

\medskip

Next we consider $\frac{d^{2j}}{d\theta^{2j}}C_{2t+2m,t+m}(\theta)$.
To this end, let us define the index set:
\begin{equation}\label{eq:bar_asymp_d}
  D_{p,q}\eqdef\{(d_1,\cdots,d_q):\,0\le d_1\le\cdots\le d_q\le p\}\;,
\end{equation}
and the following $Z$-functions:
\begin{equation}\label{eq:bar_asymp_z_def}
  Z_{p,q}(x) = \sum_{D_{p,q}}(x-d_1)^2(x-d_2)^2\cdots(x-d_q)^2\;.
\end{equation}
In what follows we write $n=t+m$ for simplicity (so $C_{2t+2m,t+m}(\theta)$ is written $C_{2n,n}(\theta)$) and because all $Z_{p,q}(\bullet)$ will take argument $n$, we shall just write $Z_{p,q}=Z_{p,q}(n)$.
By convention, we denote $Z_{p,0}=1$ for $p\ge0$ and $Z_{p,q}=0$ if $p<0$ or $q<0$.
A key step is to use induction to show that for $0\le j\le n$:
\begin{equation}\label{eq:bar_asymp_cder}
  \frac{d^{2j}C_{2n,n}}{d\theta^{2j}} = \sum_{k=0}^j(-1)^k\left(\frac{(2n)!}{(2n-j+k)!}\right)^2Z_{j-k,k}C_{2n-j+k,n-j+k}\;.
\end{equation}
When $j=0$, \cref{eq:bar_asymp_cder} reduces to:
\begin{displaymath}
  C_{2n,n} = Z_{0,0}C_{2n,n}\;,
\end{displaymath}
which is true for all $n$.
Now suppose~\cref{eq:bar_asymp_cder} holds for $j$, where $0\le j\le n-1$, we then consider the case with $j+1$.
By~\cref{lm:bar_reh_cfun_der} and the inductive assumption, there is:
\begin{align*}
  \frac{d^{2j+2}C_{2n,n}}{d\theta^{2j+2}} =&\ \frac{d^2}{d\theta^2}\left[\sum_{k=0}^j(-1)^k\left(\frac{(2n)!}{(2n-j+k)!}\right)^2Z_{j-k,k}C_{2n-j+k,n-j+k}\right] \\
  =&\ \sum_{k=0}^j(-1)^k\left(\frac{(2n)!}{(2n-j+k)!}\right)^2Z_{j-k,k}\left[(2n-j+k)^2C_{2n-j-1+k,n-j-1+k}-(n-j+k)^2C_{2n-j+k,n-j+k}\right] \\
  =&\ \sum_{k=0}^j(-1)^k\left(\frac{(2n)!}{(2n-j-1+k)!}\right)^2Z_{j-k,k}C_{2n-j-1+k,n-j-1+k} + \\
  &\ \sum_{k=1}^{j+1}(-1)^k\left(\frac{(2n)!}{(2n-j-1+k)!}\right)^2(n-j-1+k)^2Z_{j+1-k,k-1}C_{2n-j-1+k,n-j-1+k} \\
  =&\ \sum_{k=0}^{j+1}(-1)^k\left(\frac{(2n)!}{(2n-j-1+k)!}\right)^2\left[Z_{j-k,k}+(n-j-1+k)^2Z_{j+1-k,k-1}\right]C_{2n-j-1+k,n-j-1+k}\;.
\end{align*}
To complete the induction, we need to show $Z_{j+1-k,k}=Z_{j-k,k}+(n-j-1+k)^2Z_{j+1-k,k-1}$, which is not difficult:
\begin{align*}
  Z_{j+1-k,k} =&\ \sum_{D_{j+1-k,k}}(n-d_1)^2\cdots(n-d_k)^2 \\
  =&\ \sum_{D_{j-k,k}}(n-d_1)^2\cdots(n-d_k)^2 + \sum_{D_{j+1-k,k-1}}(n-d_1)^2\cdots(n-d_{k-1})^2(n-(j+1-k))^2 \\
  =&\ Z_{j-k,k} + (n-j-1+k)^2Z_{j+1-k,k-1}\;.
\end{align*}
One last step before estimating $\oname{Re}H(\pi)$ is the following estimate for $Z_{j-k,k}$, $0\le k\le j\le m-1$, which are not difficult to obtain:
\begin{equation}\label{eq:bar_asymp_z}
  Z_{j-k,k} = \sum_{D_{j-k,k}}(t+m-d_1)^2\cdots(t+m-d_k)^2 = \binom{j}{k}t^{2k}(1+\oone)\;.
\end{equation}
At last, combining~\cref{eq:bar_reh_crep}, ~\cref{eq:bar_reh_cfun_pi} and~\cref{eq:bar_asymp_cder}, and using~\cref{eq:bar_asymp_q} and~\cref{eq:bar_asymp_z}, we have:
\begin{align*}
  \oname{Re}H(\pi) &= \binom{2t+2m}{t+m}^{-2}\sum_{j=0}^{m-1}(-1)^jq_{t,m;j}\sum_{k=0}^j(-1)^k\left(\frac{(2t+2m)!}{(2t+2m-j+k)!}\right)^2Z_{j-k,k}\binom{2t+2m-j+k}{t+m-j+k} \\
  &= \binom{2t+2m}{t+m}^{-2}\sum_{j=0}^{m-1}(-1)^jq_{t,m;j}\sum_{k=0}^j(-1)^{j-k}\left(\frac{(2t+2m)!}{(2t+2m-k)!}\right)^2Z_{k,j-k}\binom{2t+2m-k}{t+m-k} \\
  &= \binom{2t+2m}{t+m}^{-1}\sum_{k=0}^{m-1}\frac{(-1)^k(t+m)!(2t+2m)!}{(2t+2m-k)!(t+m-k)!}\sum_{j=k}^{m-1}q_{t,m;j}Z_{k,j-k} \\
  &= \binom{2t+2m}{t+m}^{-1}\sum_{k=0}^{m-1}\frac{(-1)^k2(t+m)!(2t+2m)!t^{-2k-1}(1+\oone)}{(2t+2m-k)!(t+m-k)!}\left[\sum_{j=k}^{m-1}\binom{2m}{2j+1}\binom{j}{j-k}\right]\;.
\end{align*}
For each $0\le k\le m-1$, we have:
\begin{displaymath}
  \frac{(t+m)!(2t+2m)!}{(2t+2m-k)!(t+m-k)!} = (2t)^kt^k(1+\oone) = 2^kt^{2k}(1+\oone)\;,
\end{displaymath}
and it follows that:
\begin{align*}
  \oname{Re}H(\pi) &= \binom{2t+2m}{t+m}^{-1}\sum_{k=0}^{m-1}(-1)^k2^{k+1}t^{-1}(1+\oone)\sum_{j=k}^{m-1}\binom{2m}{2j+1}\binom{j}{k} \\
  &= \frac{2}{t}\binom{2t+2m}{t+m}^{-1}(1+\varepsilon)\sum_{j=0}^{m-1}\binom{2m}{2j+1}\sum_{k=0}^j(-2)^k\binom{j}{k} \\
  &= \frac{2}{t}\binom{2t+2m}{t+m}^{-1}(1+\varepsilon)\sum_{j=0}^{m-1}(-1)^j\binom{2m}{2j+1}
   = \binom{2t+2m}{t+m}^{-1}(1+\varepsilon)\frac{2}{it}\sum_{j=0}^{m-1}\binom{2m}{2j+1}i^{2j+1} \\
  &= \binom{2t+2m}{t+m}^{-1}(1+\varepsilon)\frac{2}{it}\frac{(1+i)^{2m}-(1-i)^{2m}}{2}
   = \binom{2t+2m}{t+m}^{-1}\frac{2^{m+1}}{t}\sin\frac{m\pi}{2}(1+\varepsilon)\;.
\end{align*}
Here the carrying over of $(1+\varepsilon)$ along the equalities is justified by the fact that $m$ is fixed and all summations are independent of $t$; and the last equality indicates that when $t$ is large ($\varepsilon\ll1$), the sign of $\oname{Re}H(\pi)$ is the same as $\sin\frac{m\pi}{2}$ given that $m$ is odd, which proves~\cref{thm:bar_asymp}.

\section{Relation to other Hermite type methods}
\label{sec:hermite}
The hybrird-variable discretization builds on Hermite interpolations.
To see the relation, let us consider in this section a stencil such that $l=l'$ and $r=r'$ (so $(l,r,l',r')=(l,r,l,r)$).
Let $U(x) = \int_0^xu(y)dy$ be an anti-derivative of $u(x)$, then one can interprete $\overline{u}_{\phf{j}}$ and $u_j$ as approximations to nodal value and derivative of $U$:
\begin{displaymath}
  \overline{u}_{\phf{j}} = \frac{1}{h}(U(x_{j+1})-U(x_j))\;,\quad
  u_j = U_x(x_j)\;.
\end{displaymath}
Using the primitives,~\cref{eq:prelim_ddo} can be written as:
\begin{displaymath}
  [\mathcal{D}_xu]_j = \frac{1}{h^2}\sum_{k=-l}^{r-1}\alpha_k(U_{j+k+1}-U_{j+k}) + \frac{1}{h}\sum_{k=-l}^r\beta_kDU_{j+k}\;,
\end{displaymath}
where $U_j\approx U(x_j)$ and $DU_j\approx U_x(x_j)$; and $[\mathcal{D}_xu]$ is $p$-th order if and only if for sufficiently smooth $U(x)$:
\begin{equation}\label{eq:hermite_ddo}
  U_{xx}(x_j) = u_x(x_j) = \frac{1}{h^2}\sum_{k=-l}^r\Delta\alpha_k U(x_{j+k}) + \frac{1}{h}\sum_{k=-l}^r\beta_k U'(x_{j+k}) + O(h^{p})\;,
\end{equation}
where $\Delta\alpha_k=\alpha_{k-1}-\alpha_k$, see also~\cref{eq:app_rev_err_hv}.
It follows that $\Delta\alpha_k$ and $\beta_k$ are coefficients of the balanced Hermite interpolation using nodal value and first-derivative at points $x_{j+k}$, $-l\le k\le r$.
Other stencils can be converted to an unbalanced Hermite interpolation problem, where at one of or both of the two end points only nodal value but no derivative is involved in the interpolation.
To this end, we expect the analysis developed in this paper applies to other Hermite-type method and will demonstrate this point to study a discretization builds on both nodal values and nodal derivatives.

In preparation, let use first write out the formula to compute $\Delta\alpha_k$ and $\beta_k$ using Hermite interpolation polynomials corresponding to the points $\mathcal{X}\eqdef\{kh:\,-l\le k\le r\}$.
The Lagrange interpolation polynomials for $\mathcal{X}$ are given by $\{l_k(x):\,-l\le k\le r\}$:
\begin{equation}\label{eq:hermite_lag}
  l_k(x) = \frac{L(x)}{L_k(x)}\;,\ \textrm{ where }\ 
  L(x)=\prod_{-l\le \nu\le r}(x-\nu h)\ \textrm{ and }\ L_k(x)=\frac{L(x)}{x-kh}\;.
\end{equation}
It is easy to check for all polynomials $p(x)$ of degree no more than $2l+2r+1$:
\begin{equation}\label{eq:hermite_interp}
  p(x) = \sum_{-l\le k\le r}p(x_k)h_k(x) + \sum_{-l\le k\le r}p'(x_k)g_k(x)\;,
\end{equation}
where $\{h_k(x):\,-l\le k\le r\}$ are the fundamental polynomial of the first kind defined as:
\begin{equation}\label{eq:hermite_first}
  h_k(x) = \left[1-2l_k'(kh)(x-kh)\right]l^2_k(x)\;,
\end{equation}
and $\{g_k(x):\,-l\le k\le r\}$ are the fundamental polynomial of the second kind given by:
\begin{equation}\label{eq:hermite_second}
  g_k(x) = (x-kh)l^2_k(x)\;.
\end{equation}
Setting $p(x) = Q(x-x_j)$ in~\cref{eq:hermite_interp} and comparing with~\cref{eq:hermite_ddo}, one obtains\footnote{This proves the formula in~\cref{lm:prelim_coef} in the case $l=l'$ and $r=r'$.}:
\begin{equation}\label{eq:hermite_hv_coef}
  \Delta\alpha_k = h^2h_k''(0)\;,\quad
  \beta_k = hg''_k(0)\;,\quad-l\le k\le r\;.
\end{equation}

Now let us consider a conservative discretization that has been utilized in constructing the optimally accurate baseline scheme in Hermite-WENO methods~\cite{GCapdeville:2008a,HLiu:2015a,XLi:2020a}.
Following this strategy, to solve the Cauchy problem~\cref{eq:prelim_adv} one seeks numerical solutions to cell averages of both the solution function $u$ and its spatial derivative $v=u_x$.
The semi-discretized variables are thus denoted:
\begin{equation}\label{eq:hermite_weno_sol}
  u_j(t) \approx \frac{1}{h}\int_{x_{\mhf{j}}}^{\phf{j}}u(x,t)dx\;,\quad
  v_j(t) \approx \frac{1}{h}\int_{x_{\mhf{j}}}^{\phf{j}}u_x(x,t)dx = u(x_{\phf{j}},t)-u(x_{\mhf{j}},t)\;.
\end{equation}
Note that we adopt the notaion in the HWENO literature, which is slightly different from what we used before: the cell faces are at half-grid points and the overline is omitted for simplity as all quantities are cell-averaged.
Conservative discretization of the governing equations $u_t+u_x=0$ and $v_t+v_x = (u_t+u_x)_x = 0$ is given by:
\begin{equation}\label{eq:hermite_weno_semi}
  \left\{\begin{array}{l}
    u_j' + \frac{1}{h}\left(\mathcal{F}_{\phf{j}}-\mathcal{F}_{\mhf{j}}\right) = 0\;, \\ \vspace*{-.1in} \\
    v_j' + \frac{1}{h}\left(\mathcal{H}_{\phf{j}}-\mathcal{H}_{\mhf{j}}\right) = 0\;.
  \end{array}\right.
\end{equation}
Here $\mathcal{F}_{\phf{j}}$ and $\mathcal{H}_{\phf{j}}$ are numerical fluxes given in the general form:
\begin{equation}\label{eq:hermite_weno_genflux}
  \left\{\begin{array}{l}
    \mathcal{F}_{\phf{j}} = \mathcal{F}(u_{j-l+1},u_{j-l+2},\cdots,u_{j+r},v_{j-l+1},v_{j-l+2},\cdots,v_{j+r})\;, \\ \vspace*{-.1in} \\
    \mathcal{H}_{\phf{j}} = \mathcal{H}(u_{j-l+1},u_{j-l+2},\cdots,u_{j+r},v_{j-l+1},v_{j-l+2},\cdots,v_{j+r})\;.
  \end{array}\right.
\end{equation}
Here $l$ and $r$ designate the stencil of the method, and optimally accurate fluxes are computed by finding a polynomial $P(x)$ of degree no more than $2l+2r-1$, such that:
\begin{equation}\label{eq:hermite_weno_kexact}
  \frac{1}{h}\int_{x_{\mhf{j+k}}}^{x_{\phf{j+k}}}P(x)dx = u_{j+k}\;,\quad
  \frac{1}{h}\int_{x_{\mhf{j+k}}}^{x_{\phf{j+k}}}P'(x)dx = v_{j+k}\;,\quad -l+1\le k\le r\;, 
\end{equation}
then one computes:
\begin{equation}\label{eq:hermite_weno_flux}
  \mathcal{F}_{\phf{j}} = P(x_{\phf{j}})\;,\quad
  \mathcal{H}_{\phf{j}} = P'(x_{\phf{j}})\;.
\end{equation}

To derive explicit formula of $\mathcal{F}$ and $\mathcal{H}$ and conduct stability analysis, let us define again the anti-derivative $Q(x)=\int_{x_\phf{j}}^xP(y)dy$, then the k-exactness condition~\cref{eq:hermite_weno_kexact} gives:
\begin{equation}\label{eq:hermite_weno_qval}
  Q(x_{\phf{j+k}}) = hU_k^j,\quad 
  Q'(x_{\phf{j+k}}) = hC+hV_k^j,\quad
  -l\le k\le r\;,
\end{equation}
where $C$ is a constant designating $P(x_{\phf{j}})/h$ and:
\begin{displaymath}
  U_k^j \eqdef \left\{\begin{array}{lcl}
    \sum_{\nu=1}^ku_{j+\nu}\;, & & 1\le k\le r\;, \\
    0\;, & & k = 0\;, \\
    -\sum_{\nu=1}^{-k}u_{j-\nu+1}\;, & & -l\le k\le -1\;.
  \end{array}\right.\;,\quad
  V_k^j \eqdef \left\{\begin{array}{lcl}
    \sum_{\nu=1}^kv_{j+\nu}\;, & & 1\le k\le r\;, \\
    0\;, & & k = 0\;, \\
    -\sum_{\nu=1}^{-k}v_{j-\nu+1}\;, & & -l\le k\le -1\;.
  \end{array}\right.
\end{displaymath}
On the one hand, by the Hermite interpolation theory there is a unique $Q$ of degree no more than $2l+2r+1$, such that~\cref{eq:hermite_weno_qval} are satisfied; and it is given by:
\begin{equation}\label{eq:hermite_weno_q}
  Q(x) = h\sum_{k=-l}^rU_k^jh_k(x-x_{\phf{j}})+h\sum_{k=-l}^r(C+V_k^j)g_k(x-x_{\phf{j}})\;;
\end{equation}
on the other hand, $Q$ is the anti-derivative of $P$ and thus has degree no more than $2l+2r$, it follows that $C$ must be the value which makes the $x^{2l+2r+1}$-coefficient zero, or:
\begin{displaymath}
  \sum_{k=-l}^r\frac{-2l'_k(kh)}{L^2_k(kh)}U_k^j + \sum_{k=-l}^r\frac{1}{L_k^2(kh)}(V_k^j+C) = 0\ \Rightarrow\ 
  C = -\frac{\sum_{\nu=-l}^r[V_\nu^j-2l'_\nu(\nu h)U_\nu^j]/L^2_\nu(\nu h)}{\sum_{\nu=-l}^r1/L_\nu^2(\nu h)}\;.
\end{displaymath}
Substituting $C$ in~\cref{eq:hermite_weno_q}, one computes the fluxes $\mathcal{F}_{\phf{j}}=Q'(x_{\phf{j}})$ and $\mathcal{H}_{\phf{j}}=Q''(x_{\phf{j}})$:
\begin{align*}
  \frac{\mathcal{F}_{\phf{j}}}{h} &= \sum_{k=-l}^rU_k^j\left[h'_k(0)+\frac{2l'_k(kh)\sum_{\nu=-l}^rh'_\nu(0)}{\sum_{\nu=-l}^rL^2_k(kh)/L^2_\nu(\nu h)}\right] + \sum_{k=-l}^rV_k^j\left[g'_k(0)-\frac{\sum_{\nu=-l}^rg'_\nu(0)}{\sum_{\nu=-l}^rL^2_k(kh)/L^2_\nu(\nu h)}\right] \\
  &= \sum_{k=-l}^rU_k^j\frac{2l'_k(kh)}{\sum_{\nu=-l}^rL^2_k(kh)/L^2_\nu(\nu h)} - \sum_{k=-l}^rV_k^j\frac{1}{\sum_{\nu=-l}^rL^2_k(kh)/L^2_\nu(\nu h)}\;, \\
  \frac{\mathcal{H}_{\phf{j}}}{h} &= \sum_{k=-l}^rU_k^j\left[h''_k(0)+\frac{2l'_k(kh)\sum_{\nu=-l}^rg''_\nu(0)}{\sum_{\nu=-l}^rL^2_k(kh)/L^2_\nu(\nu h)}\right] + \sum_{k=-l}^rV_k^j\left[g''_k(0)-\frac{\sum_{\nu=-l}^rg''_\nu(0)}{\sum_{\nu=-l}^rL^2_k(kh)/L^2_\nu(\nu h)}\right]\;. 
\end{align*}
Here in the calculation of $\mathcal{F}_{\phf{j}}$ we used the interpolation properties of $h_k$ and $g_k$ as well as $V_0^j=0$.
Using $U_k^j-U_k^{j-1}=u_{j+k}-u_j$ and $V_k^j-V_k^{j-1}=v_{j+k}-v_j$, we obtain the update equations for $u_j(t)$ and $v_j(t)$:
\begin{align*}
  u'_j &= -\frac{\mathcal{F}_{\phf{j}}-\mathcal{F}_{\mhf{j}}}{h} = -\sum_{k=-l}^r2c_kl'_k(kh)(u_{j+k}-u_j)+\sum_{k=-l}^rc_k(v_{j+k}-v_j)\;, \\
  v'_j &= -\frac{\mathcal{H}_{\phf{j}}-\mathcal{H}_{\mhf{j}}}{h} = -\sum_{k=-l}^r\frac{\Delta\alpha_k+2c_khl'_k(kh)b_0}{h^2}(w_{j+k}-w_j)-\sum_{k=-l}^r\frac{\beta_k-a_kb_0}{h}(v_{j+k}-v_j)\;.
\end{align*}
Here we used~\cref{eq:hermite_hv_coef}, denoted $b_0=\sum_{k=-l}^r\beta_k$, and defined the following constants:
\begin{displaymath}
  c_k = \frac{1/L^2_k(kh)}{\sum_{\nu=-l}^r1/L^2_\nu(\nu h)} 
  = \frac{\binom{l+r}{l+k}^2}{\sum_{\nu=-l}^r\binom{l+r}{l+\nu}\binom{l+r}{r-\nu}}
  = \binom{2l+2r}{l+r}^{-2}\binom{l+r}{l+k}^2\;,\quad-l\le k\le r\;.
\end{displaymath}
Furthermore, it is not difficult to compute that $hl'_k(kh) = \zeta_k^{l,r}$.

Following the von Neumann stability analysis, we make the ansatz $u_j=U(t)e^{i\kappa x_j}=U(t)e^{ij\theta}$ and $v_j=i\kappa V(t)e^{i\kappa x_j} = i\kappa V(t)e^{ij\theta}$, the ODE system reduces to:
\begin{displaymath}
  \begin{bmatrix}
    U' \\ V'
  \end{bmatrix} = -\frac{1}{h}
  \begin{bmatrix}
    \sum_k2c_k\zeta_k^{l,r}(e^{ik\theta}-1) & -i\theta\sum_kc_k(e^{ik\theta}-1) \\
    \sum_k[\Delta\alpha_k+2c_k\zeta_k^{l,r}b_0]\frac{e^{ik\theta}-1}{i\theta} & \sum_k[\beta_k-c_kb_0](e^{ik\theta}-1)
  \end{bmatrix}
  \begin{bmatrix}
    U \\ V
  \end{bmatrix}\;,
\end{displaymath}
where all summations are taken over $-l\le k\le r$.
Therefore, a necessary condition for the method to be stable is $\oname{Re}\tilde{H}(\theta)\ge0$ for all $0\le\theta\le2\pi$, where:
\begin{equation}\label{eq:hermite_trace}
  \tilde{H}(\theta) = \sum_{k=-l}^r2c_k\zeta_k^{l,r}(e^{ik\theta}-1)+\sum_{k=-l}^r\left[\beta_k-c_kb_0\right](e^{ik\theta}-1)\;.
\end{equation}

To illustrate the linkage of stability analyis of this method and previous results, we make use of the Item 2 in the proof of~\cref{thm:bar_pi}, given in~\cref{sec:bar_pi}, to prove the following instability result.
\begin{theorem}\label{thm:hermite_weno_instab}
  The semi-discretized method~\cref{eq:hermite_weno_semi} with optimally accurate fluxes and stencil $(l,r)=(t+3,t)$ is unstable for all $t\ge0$.
\end{theorem}
\begin{proof}
  First we consider general $(l,r)$ and simplify the expression for $\tilde{H}$.
  It is clear that $\sum_{k=-l}^rc_k = 1$, using in addition $\sum_{k=-l}^rc_k\zeta_k^{l,r}=0$ (proved later), there is: 
  \begin{align}
    \notag
    \tilde{H}(\theta) &= 2\sum_{k=-l}^rc_k\zeta_k^{l,r}e^{ik\theta} - 2\sum_{k=-l}^rc_k\zeta_k^{l,r} + \sum_{k=-l}^r\beta_ke^{ik\theta}-b_0\sum_{k=-l}^rc_ke^{ik\theta} + \sum_{k=-l}^r\beta_k - b_0 \\
    \label{eq:hermite_weno_trace}
    &= 2\sum_{k=-l}^rc_k\zeta_k^{l,r}e^{ik\theta}+ H(\theta) -b_0\sum_{k=-l}^rc_ke^{ik\theta}\;.
  \end{align}
  To see $\sum_{k=-l}^rc_k\zeta_k^{l,r}=0$, we apply the Hermite interpolation property to the constant function $1$ and obtain:
  \begin{displaymath}
    1 = \sum_{k=-l}^r1\cdot h_k(x)+\sum_{k=-l}^r0\cdot g_k(x) = \sum_{k=-l}^rh_k(x)\;,
  \end{displaymath}
  then $\sum_{k=-l}^rc_k\zeta_k^{l,r}=0$ comes from the fact that the $x^{2l+2r+1}$-coefficient of $\sum_{k=-l}^rh_k'(x)$ is zero.

  Then let us focus on the case $(l,r)=(t+3,t)$ and consider the three parts at $\theta=\pi$ in~\cref{eq:hermite_weno_trace} separately.
  The most difficult term to handle is actually the first one, which rewrites as:
  \begin{align*}
    \oname{Re}\sum_{k=-t-3}^tc_k\zeta_k^{t+3,t}e^{ik\pi} 
    &= \binom{4t+6}{2t+3}^{-2}\sum_{k=-t-3}^t(-1)^k(H_{t+3+k}-H_{t-k})\binom{2t+3}{t-k}^2 \\
    &= (-1)^{t+3}4\binom{4t+6}{2t+3}^{-2}\sum_{k=0}^{2t+3}(-1)^kH_k\binom{2t+3}{k}^2\;.
  \end{align*}
  Using the summation package {\it Sigma}~\cite{CSchneider:2007a,CSchneider:2021a}, one can show the identity:
  \begin{equation}\label{eq:hermite_h_identity}
    \sum_{k=0}^{2n+1}(-1)^kH_k\binom{2n+1}{k}^2 = \frac{(-1)^{n+1}2^{4n}n!n!}{(2n+1)!}
  \end{equation}
  for all non-negative integers $n$; more details are given in~\cref{app:harm}.
  Letting $n=t+1$ and substituting~\cref{eq:hermite_h_identity} into the previous equation, we obtain:
  \begin{displaymath}
    \oname{Re}\sum_{k=-t-3}^tc_k\zeta_k^{t+3,t}e^{ik\pi}  = (-1)^{2t+5}\binom{4t+6}{2t+3}^{-2}\cdot\frac{2^{4n+2}(t+1)!(t+1)!}{(2t+3)!} < 0\;.
  \end{displaymath}
  Next, by Item 2 in~\cref{sec:bar_pi}, we have:
  \begin{displaymath}
    \oname{Re}H(\pi) = -\frac{8}{t+1}\left(2+\frac{3}{t+1}+\frac{1}{t+2}+\frac{1}{t+3}\right)\binom{2t+6}{t+3}^{-1} < 0\;.
  \end{displaymath}
  Finally for the third term we have:
  \begin{align*}
    \oname{Re}\sum_{k=-t-3}^tc_k e^{ik\pi}
    &= \binom{4t+6}{2t+3}^{-2}\sum_{k=-t-3}^t(-1)^k\binom{2t+3}{t+3+k}^2 \\
    &= \binom{4t+6}{2t+3}^{-2}\sum_{k'=-t-3}^t(-1)^{-k'-3}\binom{2t+3}{t-k'}^2 = -\oname{Re}\sum_{k=-t-3}^tc_ke^{ik\pi}
  \end{align*}
  where we've used the change of variable $k'=-k-3$;
  it follows that $\oname{Re}\sum_{k=-t-3}^tc_ke^{ik\pi} = 0$.
  Summing them up, we see that $\oname{Re}\tilde{H}(\pi) < 0$ and it follows that the method is unstable.
\end{proof}

\section{Conclusions}
\label{sec:concl}
In this paper we established results on the stability barriers of a class of high-order Hermite-type discretizations of linear advection equations that are underlying the recently developed active flux methods and hybrid-variable methods.
In particular, we proved that the stencil of the discretization cannot be biased too much towards the upwind direction for stability consideration.
Despite the similarity of the result as its counterpart for finite-difference schemes, existing analytic tools for the latter do not extend to the current case; to this end we developed combinatoric tools to study the stability of these Hermite-type methods.
Then the analysis is extended to other Hermite-type methods such as those approximating nodal values and derivatives, which has been used in the construction of Hermite WENO schemes.
Future work includes proving the sharp barrier for stability of these Hermite-type methods, which is conjectured in the current work, and studying Hermite discretizations that involve more than two moments of the solutions.

\section*{Acknowledgements}
The author is supported by the U.S. National Science Foundation (NSF) under Grant DMS-2302080.

\bibliographystyle{plain}      
\bibliography{pap}

\appendix
\crefalias{section}{appendix}

\section{Existing strategies for studying stability barrier of FDMs}
\label{app:rev}
In this appendix we review two strategies that have been successfully used in establishing the stability barrier of explicit finite difference methods for advection equations, namely order stars and error functions.
We shall also discuss why they do not easily apply to the analysis of HV discretizations.

\subsection{Order stars}
\label{app:rev_os}
A first proof of the stability barrier of finite difference methods (FDM) was done using the theory of order stars~\cite{AIserles:1982a,AIserles:1983a,AIserles:1991a}.
Let the Cauchy problem~\cref{eq:prelim_adv} be discretized by a finite difference method (FDM) in space:
\begin{equation}\label{eq:app_rev_os_fdm_semi}
  u_j' + \mathcal{D}_xu_j = 0\;,\quad \mathcal{D}_xu_j \eqdef \frac{1}{h}\sum_{k=-l}^ra_ku_{j+k}\;, 
\end{equation}
then the conventional stability analysis requires:
\begin{equation}\label{eq:app_rev_os_fdm_stab}
  \mathcal{S}_{\fdm}'\eqdef\left\{\sum_{k=-l}^ra_ke^{ik\theta}:\,0<\theta<2\pi\right\}\subseteq\mathbb{C}^+\;.
\end{equation}
The theory of order stars takes a different approach and consider the set:
\begin{equation}\label{eq:app_rev_os_fdm}
  \mathcal{O}_{\fdm}\eqdef\left\{z\in\mathbb{C}:\,-\pi\le\oname{Im}z\le\pi,\ \oname{Re}(\lambda-z)<0\,,\ \lambda=\sum_{k=-l}^ra_ke^{kz}\right\}\;.
\end{equation}
On the one hand, it is easy to see that~\cref{eq:app_rev_os_fdm_stab} indicates $\mathcal{O}_{\fdm}\cap\{i\mathbb{R}\}=\varnothing$. 
On the other hand, supposing~\cref{eq:app_rev_os_fdm_semi} is $p$-th order accurate then $\lambda-z=cz^{p+1}+O(z^{p+2})$ for some nonzero constant $c$.
Intuitively, it means that the origin is adjoined by $p+1$ sectors of equal angle $\pi/(p+1)$ of $\mathcal{O}_{\fdm}$, whereas for stability, none of these sectors extend to a region that cross the imaginary axis, see~\cref{fg:app_rev_os_fdm} for two examples.
\begin{figure}\centering
  \begin{subfigure}[b]{.49\textwidth}\centering
    \includegraphics[trim=1.8in .6in 1.8in .6in, clip, width=\textwidth]{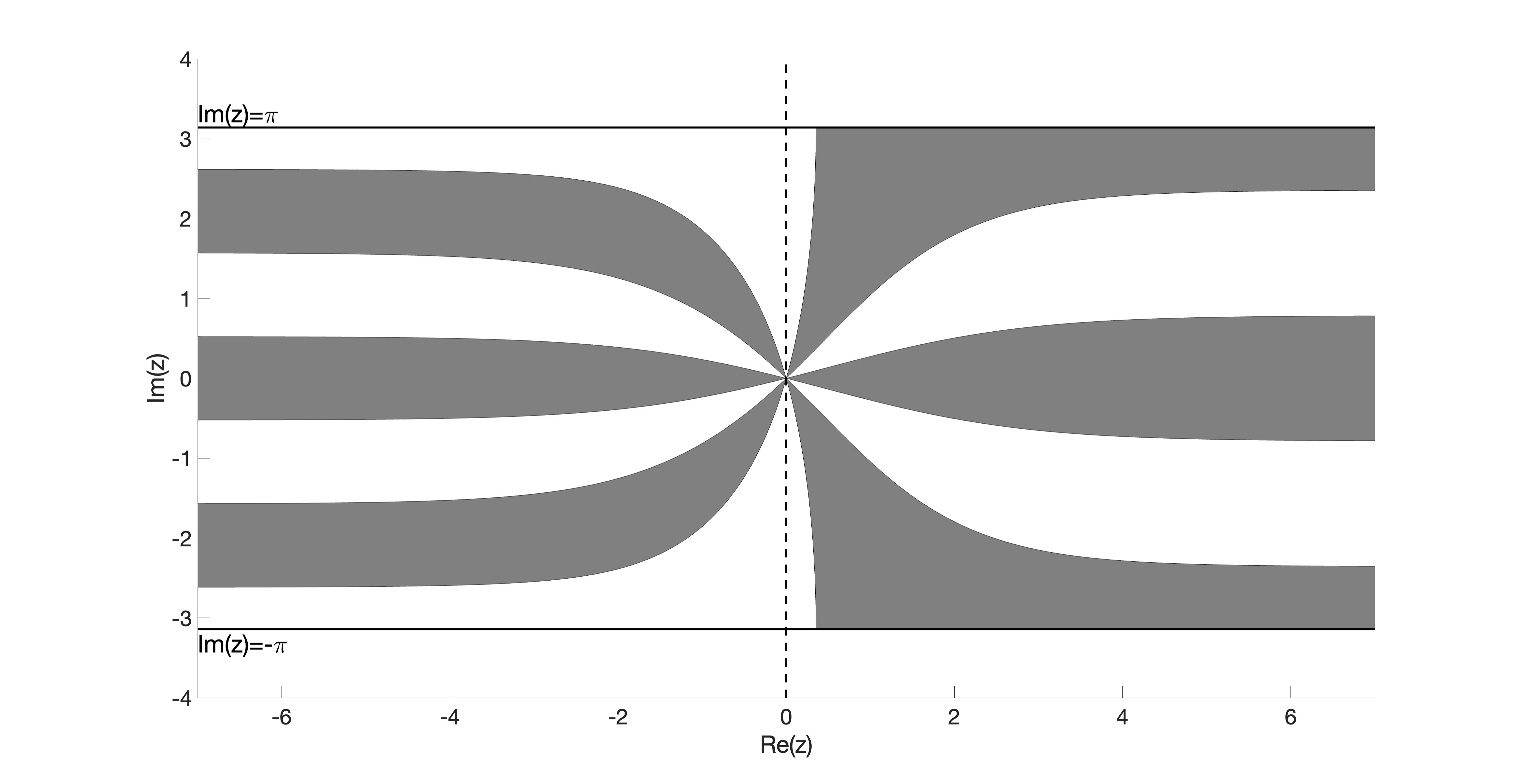}
    \caption{$\mathcal{D}_xu_j = \frac{-2u_{j-3}+15u_{j-2}-60u_{j-1}+20u_j+30u_{j+1}-3u_{j+2}}{60h}$}
    \label{fg:app_rev_os_fdm_stab}
  \end{subfigure}
  \begin{subfigure}[b]{.49\textwidth}\centering
    \includegraphics[trim=1.8in .6in 1.8in .6in, clip, width=\textwidth]{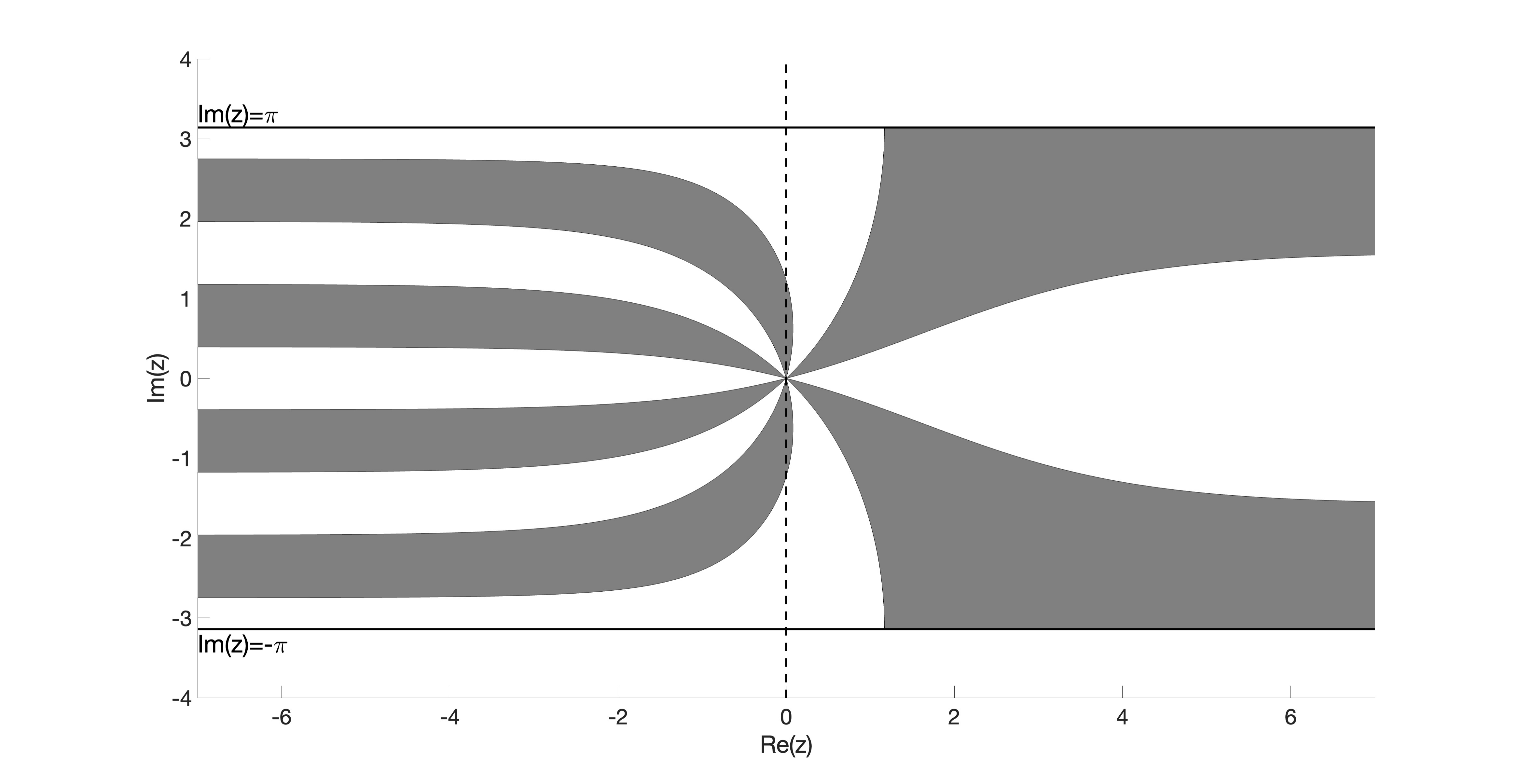}
    \caption{$\mathcal{D}_xu_j = \frac{3u_{j-4}-20u_{j-3}+60u_{j-2}-120u_{j-1}+65u_j+12u_{j+1}}{60h}$}
    \label{fg:app_rev_os_fdm_unstab}
  \end{subfigure}
  \caption{Order stars (shaded region) of fifth-order FDMs: the left one is stable with $l=3, r=2$, and the right one is unstable with $l=4, r=1$.}
  \label{fg:app_rev_os_fdm}
\end{figure}
Essentially, if a method is stable then the imaginary axis separates the ``shaded regions'' into two groups, the ones to the left of $i\mathbb{R}$ extends to $-\infty$ and number is controlled by $l$ and the ones to the right of $i\mathbb{R}$ extends to $+\infty$ and the number is controlled by $r$.
However, it is also known that each shaded sector adjoining the origin point forms an angle $\pi/(p+1)$, which tells that for a stable method $l$ could not be too different from $r$.

The beauty of the order star theory is that the proof is pure geometrical and dodges the tedious and non-trivial analysis of the positivity of the trigonometric polynomial $\oname{Re}\,\lambda = \sum_{k=-l}^ra_k\cos(k\theta)$.
However, such elegance does not carry to HV discretizations. 

First of all, one notice that $\lambda$ as defined in~\cref{eq:prelim_semi_traj} is an algebraic function of $e^{i\theta}$; thus the order star must be studied on the Riemann surface:
\begin{equation}\label{eq:app_rev_os_hv_riem}
  \mathcal{M} = \{\tilde{z}=(\lambda,z):\,-\pi\le\oname{Im}z\le\pi,\ \lambda^2-H(z/i)\lambda-F(z/i) = 0\}\;,
\end{equation}
and we define the projections $\pi(\tilde{z})=\lambda$ and $\rho(\tilde{z})=z$.
Then each $\mathcal{M}$ can be thought of as two sheets $\{z:\,-\pi\le\oname{Im}z\le\pi\}$ glued together along some cuts; and the order star is defined on this Riemann surface as:
\begin{equation}\label{eq:app_rev_os_hv}
  \mathcal{O}\eqdef\left\{\tilde{z}\in\mathcal{M}:\,\ \oname{Re}\sigma(\tilde{z})<0\,,\ \sigma(\tilde{z})=\pi(\tilde{z})-\rho(\tilde{z})\right\}\;.
\end{equation}
The Riemann surfaces and order stars corresponding to the methods~\cref{eq:prelim_ddo_43} and~\cref{eq:prelim_ddo_52} are depicted in~\cref{fg:app_rev_os_hv}.
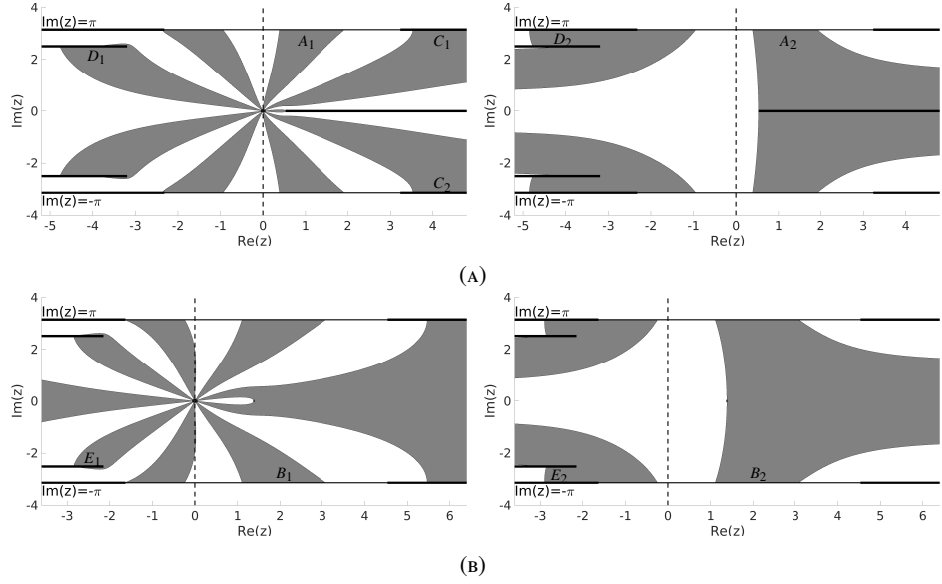
\begin{figure}\centering
  \begin{subfigure}[b]{\textwidth}\centering
    \begin{tikzpicture}
      \draw (0,0) node [inner sep=0] {\includegraphics[trim=1.2in .1in 1.8in 1.1in, clip, width=.49\textwidth]{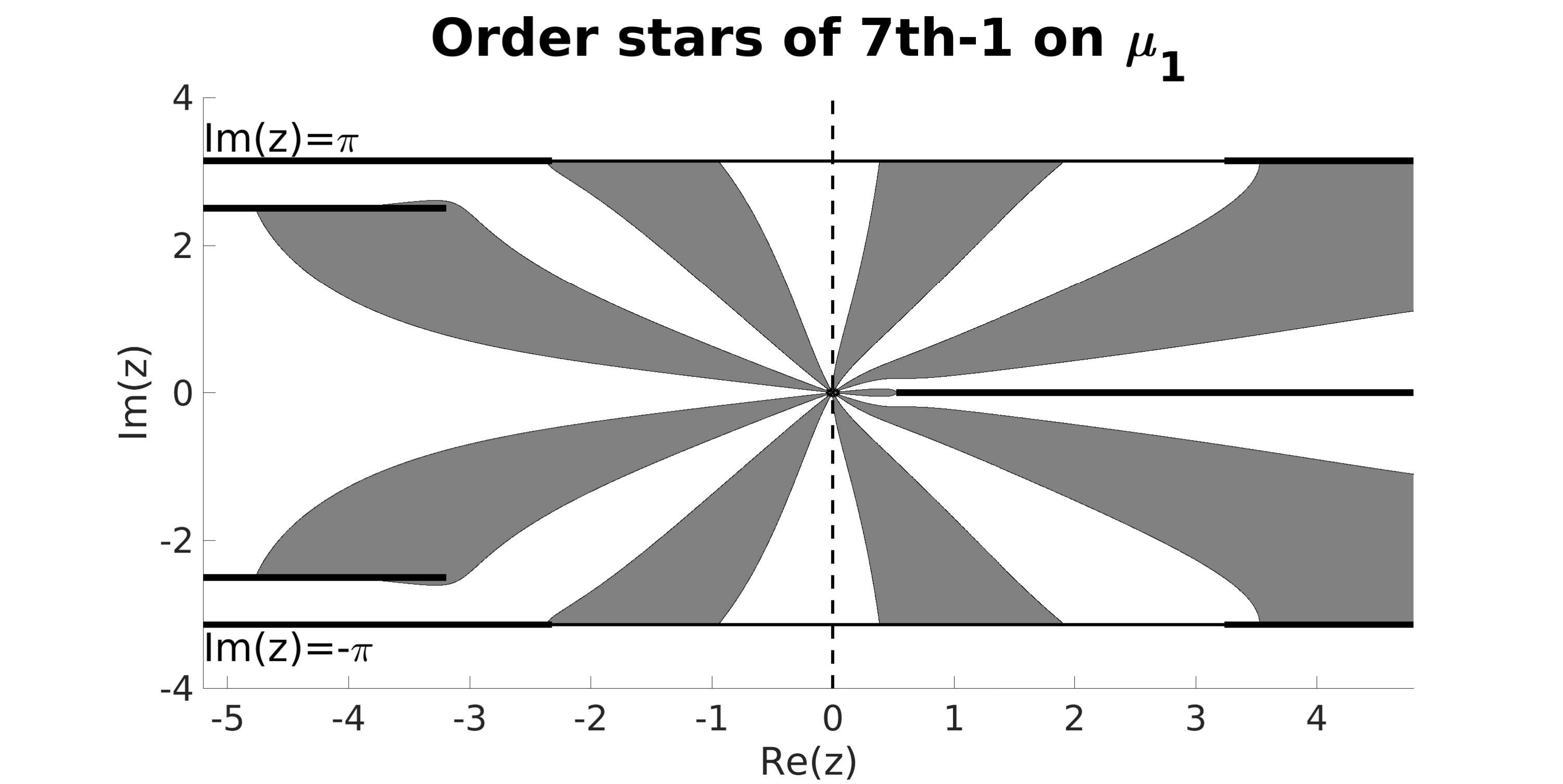}~\includegraphics[trim=1.2in .1in 1.8in 1.1in, clip, width=.49\textwidth]{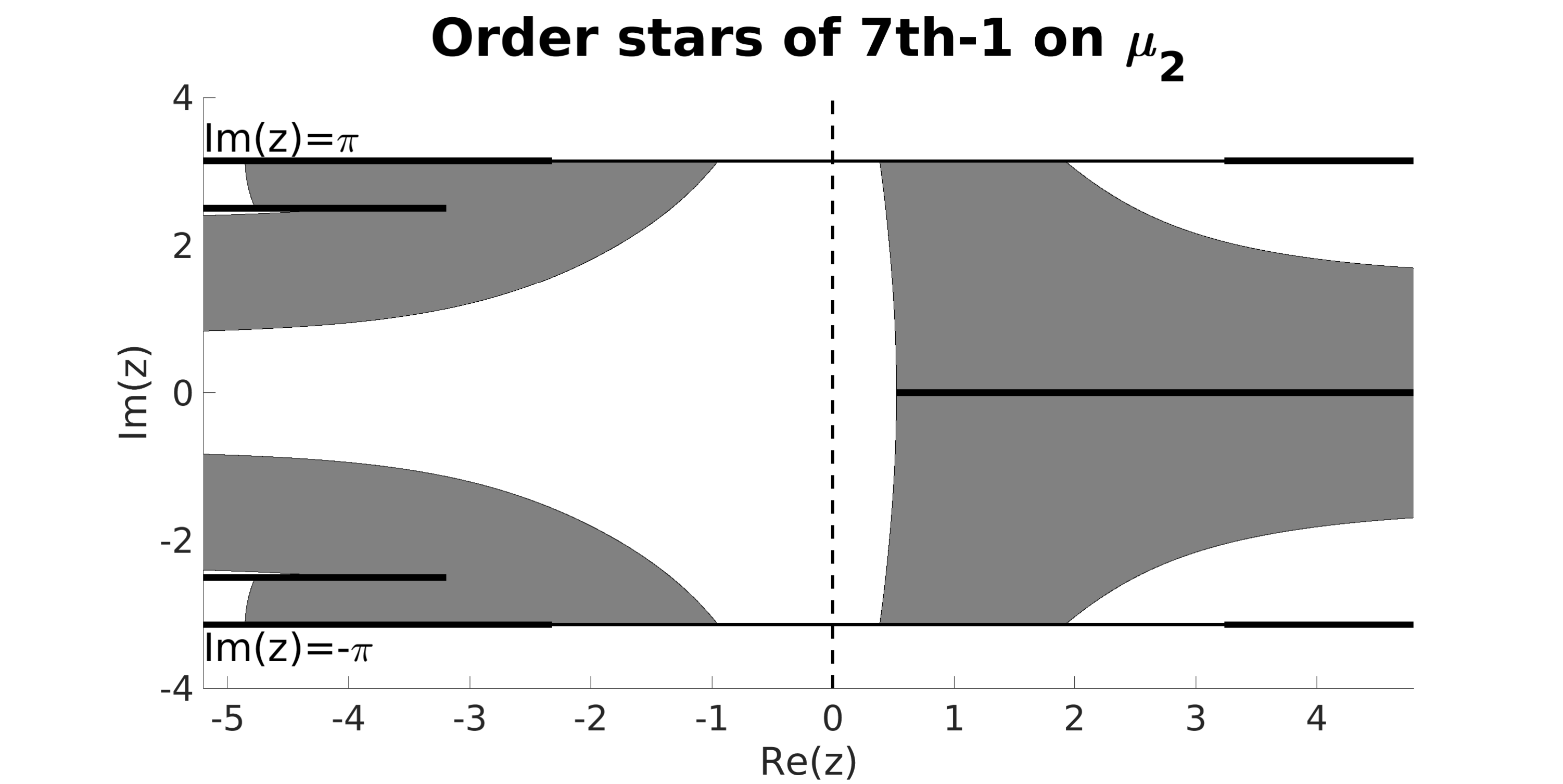}};
      \draw (-2.2,1.1) node {\tiny $A_1$};
      \draw (4.2,1.1) node {\tiny $A_2$};
      \draw (-0.4, 1.1) node {\tiny $C_1$};
      \draw (-0.4,-0.8) node {\tiny $C_2$};
      \draw (-5, 0.9) node {\tiny $D_1$};
      \draw (1.2, 1.12) node {\tiny $D_2$};
    \end{tikzpicture}
    \caption{}
    \label{fg:app_rev_os_hv_43}
  \end{subfigure}
  \begin{subfigure}[b]{\textwidth}\centering
    \begin{tikzpicture}
      \draw (0,0) node [inner sep=0] {\includegraphics[trim=1.2in .1in 1.8in 1.1in, clip, width=.49\textwidth]{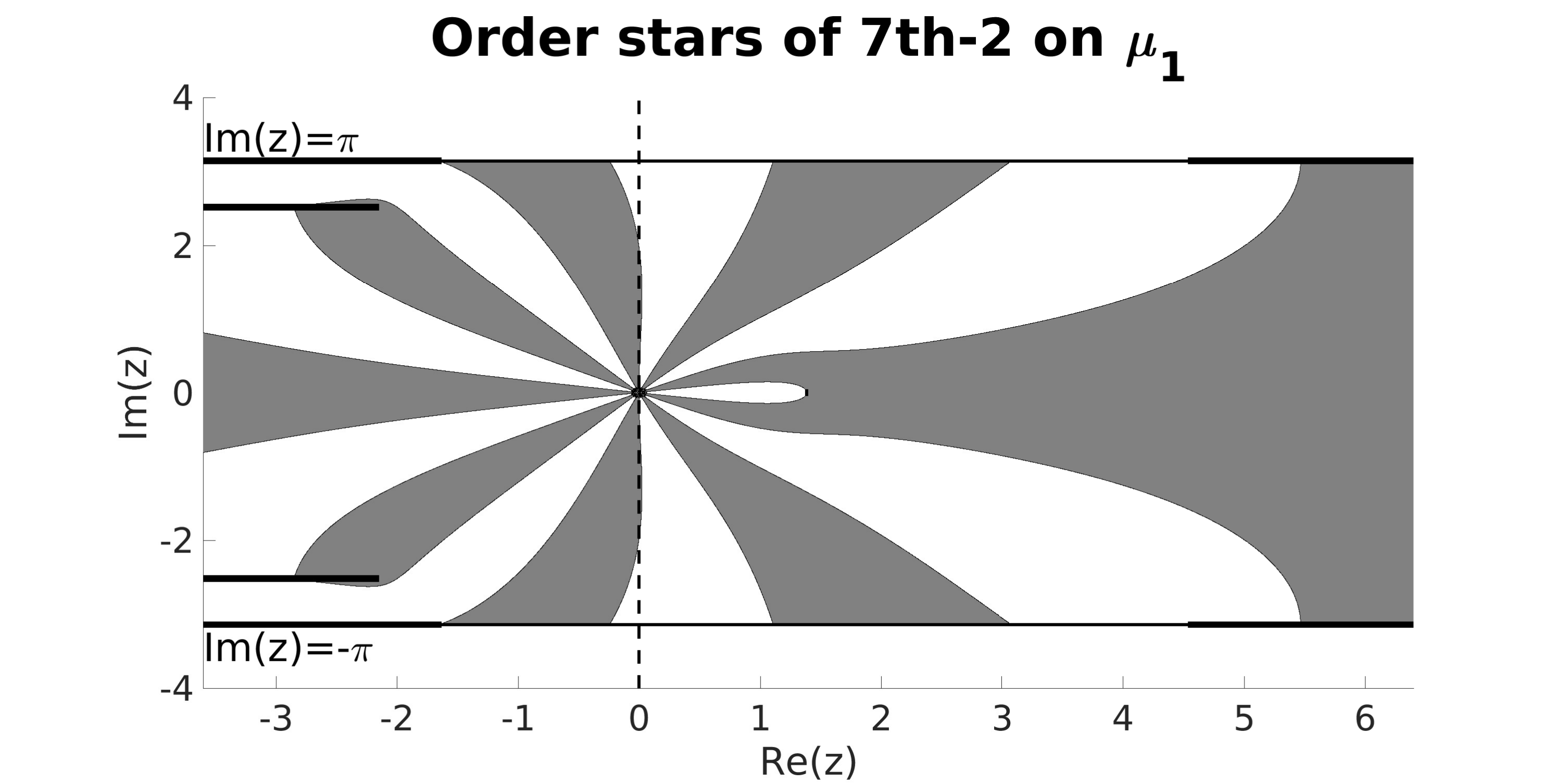}~\includegraphics[trim=1.2in .1in 1.8in 1.1in, clip, width=.49\textwidth]{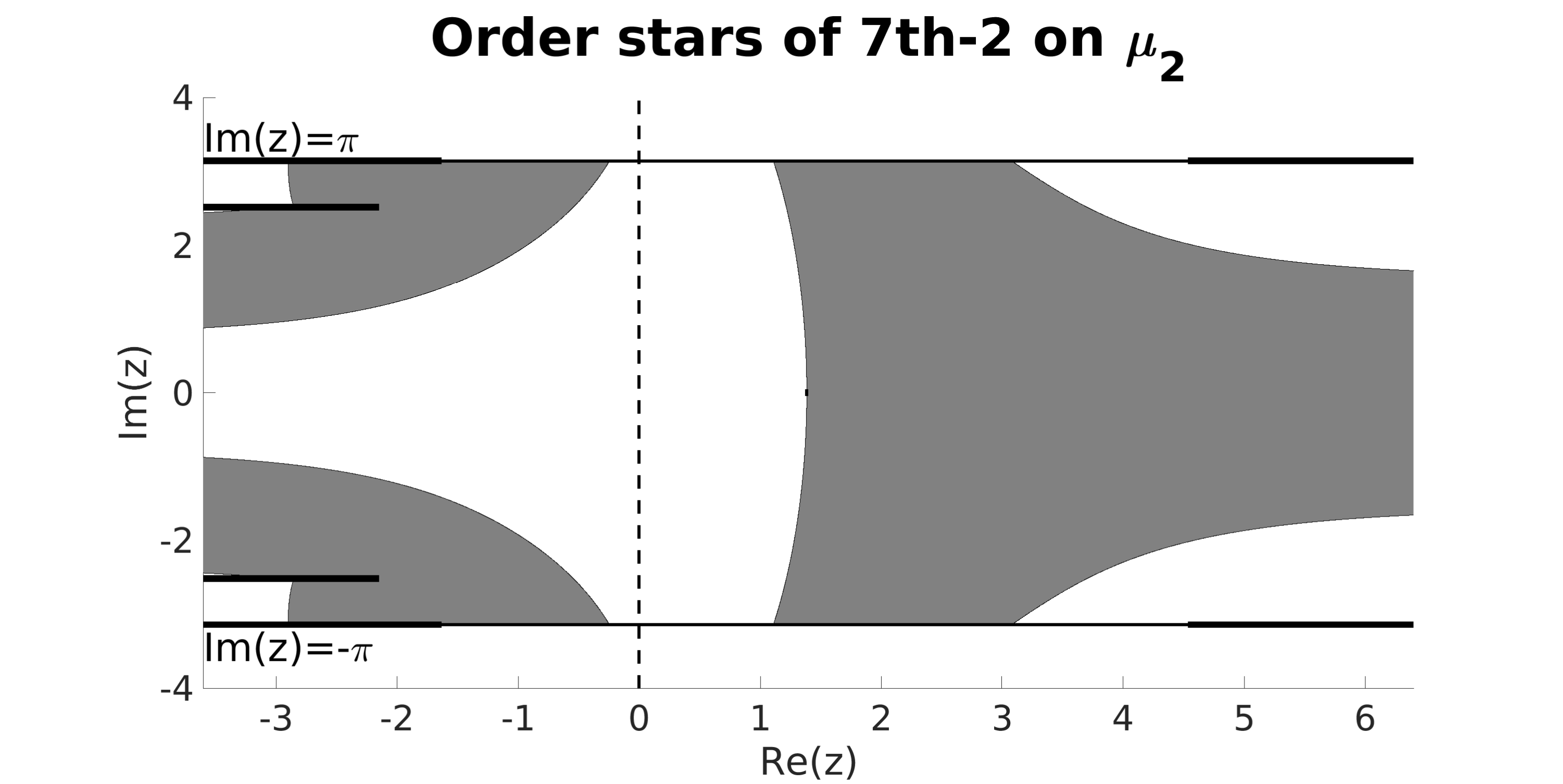}};
      \draw (-2.5,-0.8) node {\tiny $B_1$};
      \draw (3.8,-0.8) node {\tiny $B_2$};
      \draw (-5.05,-0.6) node {\tiny $E_1$};
      \draw (1.15,-0.82) node {\tiny $E_2$};
    \end{tikzpicture}
    \caption{}
    \label{fg:app_rev_os_hv_52}
  \end{subfigure}
  \caption{Order stars of two HV methods: (upper) a stable one with stencil $(4,3)$ and (lower) an unstable one with stencil $(5,2)$.
  The principal sheets are in the left panels and the second sheets in the right ones.}
  \label{fg:app_rev_os_hv}
\end{figure}
In these figures, a thick line represents a branch cut (and the end points of these lines are branch points of the Riemann surface).
To understand the geometry, let us consider a continuous trajectory $\gamma$ not passing through any branch point on $\mathcal{M}$.
Generally, a trajectory crossing $\oname{Im}z=\pm\pi$ at $x\pm i\pi$ on one sheet will enter the other sheet from the same location on this line (e.g., $A_1\leftrightarrow A_2$ and $B_1\leftrightarrow B_2$), unless the crossing point is on a branch cut in which case $\gamma$ will enter the same sheet at $x\mp i\pi$ (like $C_1\leftrightarrow C_2$).
If $\gamma$ cross a branch cut in $-\pi<\oname{Im}z<\pi$, it will enter the other sheet on the opposite side of the same branch cut (for example, $D_1\leftrightarrow D_2$ and $E_1\leftrightarrow E_2$).
It should be noted that in~\cref{fg:app_rev_os_hv_52}, there are two cuts near the positive real axis and they are very close.

The difficulty using order stars to study HV methods lies in the branch points.
Particularly, the boundary of $\mathcal{O}_{\fdm}$ only consists of infinite curves whereas that of $\mathcal{O}$ contains closed ones that enclose branch points.
While the number of infinite curves are controled by the stencil index, that of closed curves to the left and right of imaginary axis depends on the number of branch points in these regions.
However, a point $z$ is a branch point if it solves $H^2(z/i)+4F(z/i)=0$, and this is a problem no simpler than studying the conditions~\cref{eq:prelim_stab_cond} directly. 


\subsection{Integral form of the error function}
\label{app:rev_err}
A second proof of the stability barrier of FDMs by Despr\'{e}s~\cite{BDespres:2009a} studies the approximation error of $\mathcal{D}_xu_j$.
Let $\mathcal{D}_x$ in~\cref{eq:app_rev_os_fdm_semi} have the optimal order of accuracy $p=l+r$, one has:
\begin{displaymath}
  i\theta = \sum_{k=-l}^ra_ke^{ik\theta}+O(\theta^{p+1})\;,
\end{displaymath}
and by writing $e^{-i\theta}=1+z$:
\begin{displaymath}
  (1+z)^r\ln(1+z) = P(z) + O(z^{p+1})\,,\quad P(z) \eqdef  -\sum_{k=-l}^ra_k(1+z)^{r-k}.
\end{displaymath}
As $P(z)$ has degree $p$, it is the $p$-th Taylor polynomial of $S(z) = (1+z)^r\ln(1+z)$ and by the Taylor's theorem:
\begin{equation}\label{eq:app_rev_err_fdm}
  S(z) = P(z) + \int_0^1\frac{(1-t)^pS^{(p+1)}(tz)z^{p+1}}{p!}dt = P(z) + \int_0^1\frac{(-1)^l(1-t)^pr!l!z^{p+1}}{p!(1+tz)^{l+1}}dt\;,
\end{equation}
where $S^{(p+1)}(z)$ is the $(p+1)$-th derivative of $S(z)$.
Using the method of contradiction, Despr\'{e}s shows by elementary calculus that the real part of the integral term in~\cref{eq:app_rev_err_fdm} cannot be positive for all $\theta$ if $l\ge r+3$, which is required by the stability.

In contrast, the same strategy does not apply to hybrid-variable discretizations, mainly due to a lacking of a general theory of algebraic approximation to non-polynomials (Taylor series produce polynomial approximations to such functions).
Specifically, one can show that an optimally accurate method defined by~\cref{eq:prelim_ddo} gives:
\begin{displaymath}
  i\theta = \frac{1}{i\theta}F(\theta)+H(\theta)+O(\theta^{p+1})\;,
\end{displaymath}
where $G$ and $F$ are given in~\cref{eq:prelim_semi_mat} and below~\cref{eq:prelim_semi_eigset}, respectively.
With the same change of variable, one gets for $S(z)=(1+z)^r\ln(1+z)$:
\begin{equation}\label{eq:app_rev_err_hv}
  \left[S(z)\right]^2 = S(z)\sum_{k=-l'}^{r'}\beta_k(1+z)^{r-k}+\sum_{k=-l}^r\Delta\alpha_k(1+z)^{2r-k} + O(z^{p+2})\;,
\end{equation}
where $\Delta\alpha_k = \alpha_{k-1}-\alpha_k$, $-l\le k\le r$, and $\alpha_{-l-1}=\alpha_r=0$.
Thus the analysis along this line is the lacking of an exact formula for the $O(z^{p+2})$ error term in~\cref{eq:app_rev_err_hv}.

\section{The proof of remaining three cases in~\cref{thm:bar_sqrt}}
\label{app:sqrt}
In this section we provide the key steps in the remaining three cases of the proof to~\cref{thm:bar_sqrt}.

\medskip

\noindent
{\bf Case 2: $L=2t+2m+1$, $R=2t$, where $t\ge0$ and $m\ge0$}.

\smallskip

In this case the four-index stencil is $(l,r,l',r')=(t+m+1,t,t+m,t)$, then following the same strategy for Case 1 in~\cref{sec:bar_sqrt}, stability of the method indicates for all $1\le k\le t+m$:
\begin{displaymath}
  \sum_{j=1}^k\left[\ln\left(1+\frac{m-k}{t+j}\right)+\ln\left(1+\frac{m-k+1}{t+j}\right)\right]<\ln\left(1+\sum_{j=1}^m\frac{4k}{t+j}+\frac{2k}{t+m+1}\right)\,;
\end{displaymath}
and setting $k=1$ one gets $m\le 2t+3$.

Similar to~\cref{eq:bar_sqrt_case1_halfm_simp}, picking $k=\floor{m/2}$ one obtains the inequality:
\begin{displaymath}
  \left(1+\frac{k}{t+1}\right)^k\left(1+\frac{k}{t+k}\right)^k < 1 + \frac{4k^2+3k}{t+1} + \frac{4k^2+3k}{t+k}\;.
\end{displaymath}
Then using $k\le t+1$, there is:
\begin{align*}
  &\ 2^{\frac{k^2}{t+1}+\frac{k^2}{t+k}} < 1 + \frac{4k^2+3k}{t+1} + \frac{4k^2+3k}{t+k} \le 1 + \frac{7k^2}{t+1} + \frac{7k^2}{t+k}\ 
  \Rightarrow\ \frac{2k^2}{t+k} < \frac{k^2}{t+1}+\frac{k^2}{t+k} < \frac{21}{4} \\ 
  \Rightarrow&\ k<\frac{21}{16}+\sqrt{\frac{21}{8}t+\frac{441}{256}}
  \ \Rightarrow\ L-R\le4k+3<\frac{33}{4}+\sqrt{21R+\frac{441}{16}}<9+\sqrt{21R+49}\;.
\end{align*}
Combining with $m\le 2t+3$, one has: $L-R\le\min(2R+7,9+\sqrt{21R+49})$.

\medskip

\noindent
{\bf Case 3: $L=2t+2m+1$, $R=2t+1$, where $t\ge0$ and $m\ge1$}.

\smallskip

In this case the four-index stencil is $(l,r,l',r')=(t+m+1,t+1,t+m,t)$, and the stability of the method indicates for all $1\le k\le t+m$:
\begin{displaymath}
  \sum_{j=1}^k\left[\ln\left(1+\frac{m-k}{t+1+j}\right)+\ln\left(1+\frac{m-k}{t+j}\right)\right]<\ln\left(1+\frac{2k}{t+1}+\sum_{j=2}^m\frac{4k}{t+j}+\frac{2k}{t+m+1}\right)\,;
\end{displaymath}
and setting $k=1$ one gets $m\le 2t+3$.

Then picking $k=\floor{m/2}$ gives rise to the inequality:
\begin{displaymath}
  \left(1+\frac{k}{t+1}\right)^k\left(1+\frac{k}{t+k+1}\right)^k < 1 + \frac{4k^2+2k}{t+1} + \frac{4k^2+2k}{t+k+1}\;.
\end{displaymath}
Then using $k\le t+1$, there is:
\begin{align*}
  &\ 2^{\frac{k^2}{t+1}+\frac{k^2}{t+k+1}} < 1 + \frac{4k^2+2k}{t+1} + \frac{4k^2+2k}{t+k+1} \le 1 + \frac{6k^2}{t+1} + \frac{6k^2}{t+k+1}\ 
  \Rightarrow\ \frac{2k^2}{t+k+1} < \frac{k^2}{t+1}+\frac{k^2}{t+k+1} < 5 \\ 
  \Rightarrow&\ k<\frac{5}{4}+\sqrt{\frac{5}{2}t+\frac{65}{16}}
  \ \Rightarrow\ L-R\le4k+2<7+\sqrt{20R+45}<9+\sqrt{21R+49}\;.
\end{align*}
Combining with $m\le 2t+3$, one has: $L-R\le\min(2R+6,9+\sqrt{21R+49})$.

\medskip

\noindent
{\bf Case 4: $L=2t+2m+2$, $R=2t+1$, where $t\ge0$ and $m\ge0$}.

\smallskip

In this case the four-index stencil is $(l,r,l',r')=(t+m+1,t+1,t+m+1,t)$, and the stability of the method indicates for all $1\le k\le t+m$:
\begin{displaymath}
  \sum_{j=1}^k\left[\ln\left(1+\frac{m-k}{t+1+j}\right)+\ln\left(1+\frac{m-k+1}{t+j}\right)\right]<\ln\left(1+\frac{2k}{t+1}+\sum_{j=2}^{m+1}\frac{4k}{t+j}\right)\,;
\end{displaymath}
and setting $k=1$ one gets $m\le 2t+3$.

Letting $k=\floor{m/2}$ again there is:
\begin{displaymath}
  \left(1+\frac{k}{t+1}\right)^k\left(1+\frac{k}{t+k+1}\right)^k < 1 + \frac{4k^2+3k}{t+1} + \frac{4k^2+3k}{t+k+1}\;.
\end{displaymath}
Then using $k\le t+1$ we obtain:
\begin{align*}
  &\ 2^{\frac{k^2}{t+1}+\frac{k^2}{t+k+1}} < 1 + \frac{4k^2+3k}{t+1} + \frac{4k^2+3k}{t+k+1} \le 1 + \frac{7k^2}{t+1} + \frac{7k^2}{t+k+1}\ 
  \Rightarrow\ \frac{2k^2}{t+k+1} < \frac{k^2}{t+1}+\frac{k^2}{t+k+1} < \frac{21}{4} \\ 
  \Rightarrow&\ k<\frac{21}{16}+\sqrt{\frac{21}{8}t+\frac{1113}{256}}
  \ \Rightarrow\ L-R\le4k+3<\frac{33}{4}+\sqrt{21R+\frac{777}{16}}<9+\sqrt{21R+49}\;.
\end{align*}
Combining with $m\le 2t+3$, one has: $L-R\le\min(2R+7,9+\sqrt{21R+49})$.

\section{Proofs of combinatoric results}
\label{app:comb}
In this appendix we prove the combinatoric results in~\cref{sec:bar_reh}.
Recall that given a polynomial $P(x)$, we denote the coefficient of $x^n$ by $[P]_n$ or $[P(x)]_n$, whichever suits the context better.
Most of the results in~\cref{sec:bar_reh} can be derived by studying the following more general functions with three indices:
\begin{align*}
  &C_{m_1,m_2;n}(\theta) = \sum_k\binom{m_1}{n+k}\binom{m_2}{n-k}\cos(k\theta)\;, 
  &&S_{m_1,m_2;n}(\theta) = \sum_k\binom{m_1}{n+k}\binom{m_2}{n-k}\sin(k\theta)\;, \\
  &c_{m_1,m_2;n}(\theta) = \sum_k\binom{m_1}{n+k}\binom{m_2}{n+1-k}\cos(k\theta)\;, 
  &&s_{m_1,m_2;n}(\theta) = \sum_k\binom{m_1}{n+k}\binom{m_2}{n+1-k}\sin(k\theta)\;.
\end{align*}
We adopt the convention that when $(\theta)$ is omitted, the symbol such as $C_{m_1,m_2;n}$ means the function $C_{m_1,m_2;n}(\theta)$.
Then the function $C_{m,n}$ defined in~\cref{eq:bar_reh_cfun} equals $C_{m,m;n}$.
Taking derivatives of these functions, it is elementary to compute that:
\begin{subequations}\label{eq:app_comb_aux_der}
  \begin{align}
    C'_{m_1,m_2;n} &= -\sum_k\binom{m_1}{n+k}\binom{m_2}{n-k}[(n+k)-n]\sin(k\theta) = -m_1s_{m_1-1,m_2;n-1}+nS_{m_1,m_2;n}\;, \\
    S'_{m_1,m_2;n} &= \sum_k\binom{m_1}{n+k}\binom{m_2}{n-k}[(n+k)-n]\cos(k\theta) = m_1c_{m_1-1,m_2;n-1}-nC_{m_1,m_2;n}\;, \\
    \notag
    s'_{m_1,m_2;n} &= \sum_k\binom{m_1}{n+k}\binom{m_2}{n+1-k}[(n+1)-(n+1-k)]\cos(k\theta) \\
                           &= (n+1)c_{m_1,m_2;n}-m_2C_{m_1,m_2-1;n}\;.
  \end{align}
\end{subequations}


\medskip 

To prove~\cref{lm:bar_reh_cfun_alt}, we define the $2m$-degree polynomial:
\begin{equation}\label{eq:app_comb_alt_aux}
  f(x) = (1+e^{i\theta}x)^m(1+x)^m\;,
\end{equation}
then its coefficient for $x^{2n}$ is given by:
\begin{align}
  \notag
  [f]_{2n} &= \sum_k[(1+e^{i\theta}x)^{m}]_{k}[(1+x)^{m}]_{2n-k} = \sum_k\binom{m}{k}\binom{m}{2n-k}e^{ik\theta} \\
  \label{eq:app_comb_alt_aux_even}
  &= \sum_k\binom{m}{n+k}\binom{m}{n-k}e^{i(n+k)\theta} 
   = e^{in\theta}\left[C_{m,m;n}(\theta)+iS_{m,m;n}(\theta)\right]\;.
\end{align}
%

Next, we compute $[f]_{2n}$ in an alternative way:
\begin{align*}
  [f]_{2n} &= \left[\left((1+x)(1+e^{i\theta}x)\right)^m\right]_{2n}
  = \left[\left(\sin^2\frac{\theta}{2}+\left(\cos\frac{\theta}{2}+e^{i\frac{\theta}{2}}x\right)^2\right)^m\right]_{2n} \\
  &= \sum_k\binom{m}{m-k}\cdot\sin^{2(m-k)}\frac{\theta}{2}\cdot\binom{2k}{2n}\cos^{2(k-n)}\frac{\theta}{2}\cdot e^{in\theta}\;.
\end{align*}
Then~\cref{eq:bar_reh_cfun_alt} is obtained by comparing it with~\cref{eq:app_comb_alt_aux_even}.

\medskip
 

Next to show~\cref{lm:bar_reh_cfun_der}, we use~\cref{eq:app_comb_aux_der} to have:
\begin{align*}
  C''_{m,n} &= C''_{m,m;n} = -ms'_{m-1,m;n-1}+nS'_{m,m;n} \\
  &= -m(nc_{m-1,m;n-1}-mC_{m-1,m-1;n-1})+n(mc_{m-1,m;n-1}-nC_{m,m;n}) \\
  &= m^2C_{m-1,m-1;n-1}-n^2C_{m,m;n} = m^2C_{m-1,n-1}-n^2C_{m,n}\;.
\end{align*}
This is precisely~\cref{eq:bar_reh_cfun_der} in~\cref{lm:bar_reh_cfun_der}.


\section{Proof of~\cref{eq:hermite_h_identity} using {\it Sigma}}
\label{app:harm}
In this section we show the identity~\cref{eq:hermite_h_identity} using the summation package {\it Sigma}~\cite{CSchneider:2007a,CSchneider:2021a}.
The results generated by {\it Sigma} are displayed in~\cref{fg:app_harm_sigma}, and we explain these verifiable steps briefly next.
\begin{figure}\centering
  \includegraphics[width=.9\textwidth]{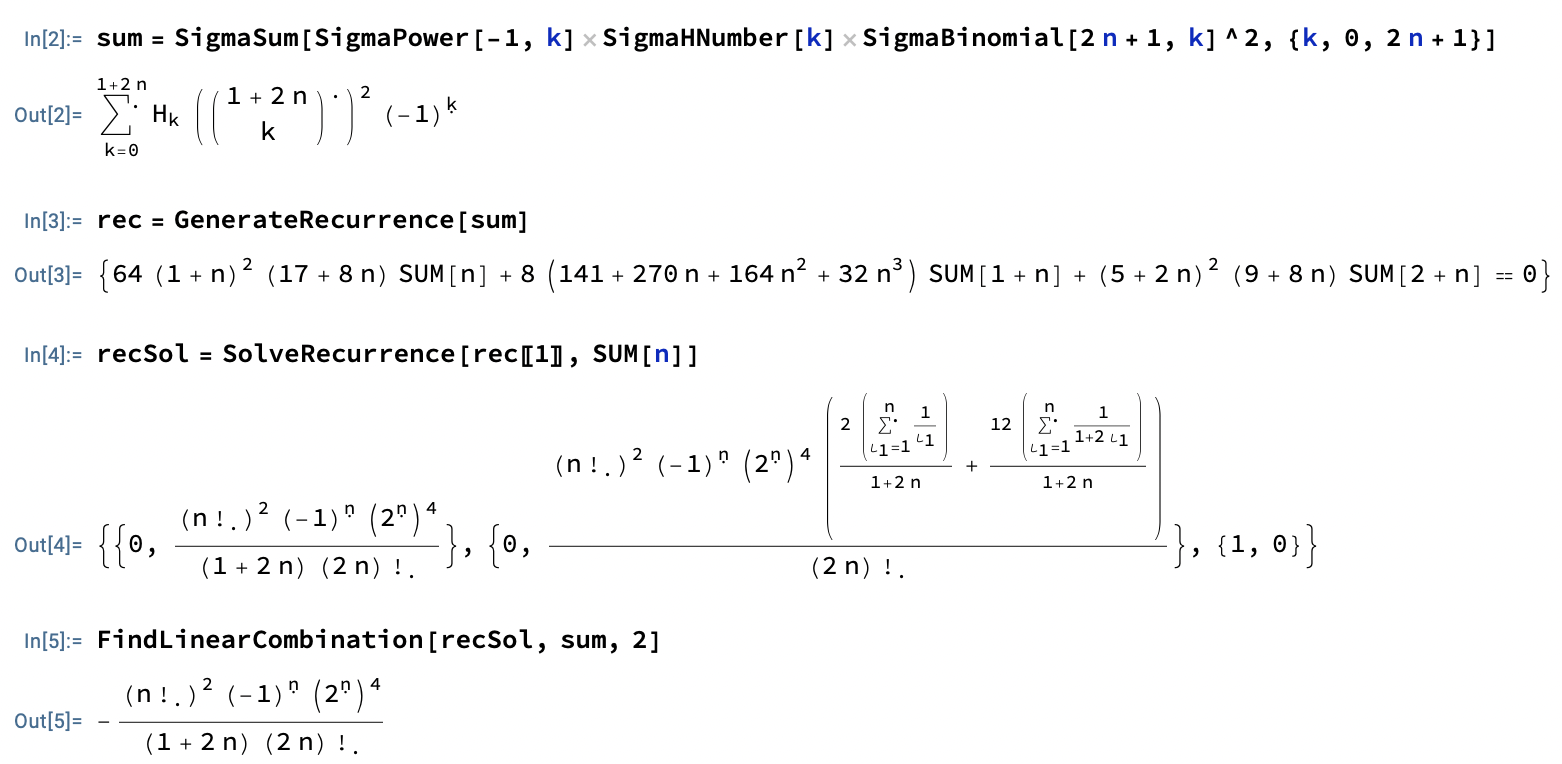}
  \caption{Results generated by {\it Sigma}.}
  \label{fg:app_harm_sigma}
\end{figure}
First, we represent the left hand side of~\cref{eq:hermite_h_identity} as:
\begin{equation}\label{eq:app_harm_ct}
  S_n \eqdef \sum_{k=0}^{2n+1}(-1)^kH_k\binom{2n+1}{k}^2 = \sum_{k=0}^{n}f_{n,k}\;,\quad
  f_{n,k}\eqdef\sum_{j=2k}^{2k+1}(-1)^jH_j\binom{2n+1}{j}^2\;.
\end{equation}
Then by the method of creative telescoping~\cite{DZeilberger:1991a}, {\it Sigma} computes the reccursive relation of the $f$-functions:
\begin{equation}\label{eq:app_harm_recur_gen}
  g_{n,k+1}-g_{n,k} = c_0(n)f_{n,k}+c_1(n)f_{n+1,k}+\cdots+c_\delta f_{n+\delta,k}
\end{equation}
for some small non-negative integer $\delta$ and function $g$; then taking the sum of~\cref{eq:app_harm_recur_gen} over all $0\le k\le n$, one gets:
\begin{equation}\label{eq:app_harm_recur}
  c_0(n)S_n +c_1(n)S_{n+1}+\cdots+c_{\delta}(n)S_{n+\delta} = q(n)\;,
\end{equation}
where $q(n)=g_{n,n+1}-g_{n,0}+c_1(n)f_{n+1,n+1}+c_2(n)[f_{n+2,n+1}+f_{n+2,n+2}]+\cdots+c_\delta(n)[f_{n+\delta,n+1}+\cdots+f_{n+\delta,n+\delta}]$.
The recursive formula~\cref{eq:app_harm_recur_gen} corresponding to our formula~\cref{eq:app_harm_ct} computed by this algorithm has $\delta=2$ and is given by:
\begin{equation}\label{eq:app_harm_ct_recur}
  64(n\!+\!1)^2(8n\!+\!17)S_n + 8(32n^3\!+\!164n^2\!+\!270n\!+\!141)S_{n+1} + (8n\!+\!9)(2n\!+\!5)^2S_{n+2} = 0\;,
\end{equation}
see also \textrm{Out[3]} in~\cref{fg:app_harm_sigma}.

Next, one solves for the general solution of the recurrence relation~\cref{eq:app_harm_recur}:
\begin{equation}\label{eq:app_harm_recur_sol}
  S_n = k_0h_0(n) + k_1h_1(n) + \cdots + k_dh_d(n) + p(n)\;,
\end{equation}
where $h_0(n),\cdots,h_d(n)$ are linearly independent solutions to the homogeneous problem\footnote{In the special case $c_0(n)$ and $c_\delta(n)$ have finite many zeros, $d=\delta$; but in general $d$ could be different from $\delta$.} and $p(n)$ is a particular solution.
For our problem, this is computed as:
\begin{equation}\label{eq:app_harm_ct_recur_sol}
  S_n = k_0\cdot\frac{(-1)^n2^{4n}(n!)^2}{(2n+1)!} + k_1\cdot\frac{(-1)^n2^{4n}(n!)^2(2H_n+12(H_{2n+1}-1-\frac{1}{2}H_n))}{(2n+1)!}\;,
\end{equation}
see \textrm{Out[4]} in~\cref{fg:app_harm_sigma}.

Lastly, the $k$-coefficients are found by fitting the general solution with the first a few values of $S_n$, which eventually gives rise to $k_0=-1$ and $k_1=0$, or~\cref{eq:hermite_h_identity} as we need.
This is achieved in {\it Sigma} using the last line as shown in~\cref{fg:app_harm_sigma}.


\end{document}